\documentclass[12pt]{amsart}

\usepackage{amsmath, amssymb, amsthm}
\usepackage{hyperref}
\usepackage{mathrsfs}
\usepackage{tikz}
\usepackage{enumerate}
\usepackage[shortlabels]{enumitem}
\usepackage[english]{babel}
\usepackage{graphicx}
\usepackage{wrapfig}
\usepackage{xcolor}
\usepackage{bbm}
\usepackage{mathtools}
\usepackage{comment}
\usepackage{thmtools}
\usepackage{url}

\usepackage[backend=biber,
    style=alphabetic,
    maxbibnames=99,
    giveninits=true,
    doi=false,
    url=false,
    isbn=false,
    eprint=true]{biblatex}
\usepackage{csquotes} 
\DeclareFieldFormat{title}{\mkbibemph{#1}}

\DeclareFieldFormat[misc]{title}{\mkbibemph{#1}}

\DeclareFieldFormat[article,inbook,incollection,inproceedings,unpublished]{title}{\mkbibemph{#1}}

\def\bN{\mathbb{N}}

\def\bR{\mathbb{R}}

\def\bP{\mathbb{P}}
\def\bS{\mathbb{S}}
\def\bE{\mathbb{E}}
\def\dd{\mathrm{d}}
\def\bA{\mathbb{A}}

\def\dim{\mathrm{dim}}

\DeclareMathOperator{\SL}{SL}
\DeclareMathOperator{\Lip}{Lip}
\DeclareMathOperator{\prob}{Prob}
\DeclareMathOperator{\supp}{supp}

\DeclareMathOperator{\NF}{NF}
\DeclareMathOperator{\e}{\mathrm{e}}

\def\cA{\mathcal{A}}

\def\cE{\mathcal{E}}
\def\cF{\mathcal{G}}
\def\cI{\mathcal{I}}
\def\cL{\mathcal{L}}
\def\cN{\mathcal{N}}

\def\cS{\mathcal{S}}
\def\cT{\mathcal{T}}

\newtheorem{theoremA}{Theorem}

\newtheorem{theorem}{Theorem}
\newtheorem{corollary}{Corollary}
\newtheorem{proposition}{Proposition}
\newtheorem{lemma}{Lemma}

\newtheorem{example}{Example}

\newtheorem{lema}{Lemma}[section]

\declaretheoremstyle[
  headfont=\normalfont\itshape,
  bodyfont=\itshape,            
  spaceabove=10pt,
  spacebelow=10pt,
  headpunct={.},                
  postheadspace=1em
]{myremarkstyle}

\declaretheorem[style=myremarkstyle]{remark}

\DeclareMathSymbol{\varnothing}{\mathord}{AMSb}{"3F}
\title[Concentration inequalities for RDSs]{Concentration inequalities for\\ random dynamical systems}

\author[G. Salcedo]{Graccyela Salcedo}
\address{
Department of Mathematics. PUC--Rio. Rio de Janeiro 22451--900, Brazil
}
\email{salcedograccy@esp.puc-rio.br}

\begin{document}

\begin{abstract}
We establish concentration inequalities for random dynamical systems (RDSs), assuming that the observables of interest are separately Lipschitz. Under a weak average contraction condition, we obtain deviation bounds for several random quantities, including time-average synchronization, deviations of empirical measures from the stationary measure, Birkhoff sums, and correlation dimension estimators. We present concrete classes of RDSs to which our main results apply, such as finitely supported diffeomorphisms on the circle and projective systems induced by linear cocycles. In both cases, we obtain concentration inequalities for finite-time Lyapunov exponents.
\end{abstract}

\begin{thanks}{The author expresses her profound gratitude to Jean-René Chazottes for his unparalleled insights, which have greatly influenced all the results presented in this article. The author thanks Sandro Gallo and Daniel Y. Takahashi for the insightful mathematical discussions that strongly motivated the writing of this article.\\
This paper was written during the author’s one-year postdoctoral position at CPHT, CNRS, École Polytechnique, Institut Polytechnique de Paris, 91120 Palaiseau, France.
}\end{thanks}

\subjclass[2000]{}

\maketitle
\tableofcontents

\section{Introduction}

Concentration inequalities are pivotal tools in probability theory, providing bounds on the probability that a random variable, which depends (in a smooth way) on many independent random variables but not too much on any of them,  deviates from a central value, such as its expected value. These inequalities have been instrumental in understanding the fluctuations and stability properties of stochastic processes. In the realm of dynamical systems, particularly those exhibiting chaotic behavior, concentration inequalities offer insights into the statistical properties of orbits and the robustness of time averages. Unlike classical limit theorems such as the Central Limit Theorem or large deviations, which are asymptotic and typically apply only to Birkhoff sums, concentration inequalities provide nonasymptotic bounds that remain valid for a broad class of observables—including nonlinear and implicitly defined ones—provided they satisfy a mild regularity condition such as a separate Lipschitz property.

The application of concentration inequalities to dynamical systems, where there is no more independence, was notably advanced by Collet, Martinez, and Schmitt, who established exponential inequalities for dynamical measures associated with expanding maps of the interval, see \cite{ColMarSch:2002}. Their work laid the groundwork for subsequent studies exploring the statistical behavior of non-uniformly hyperbolic systems.
Subsequently, Chazottes, Collet and Schmitt \cite{ChaPie:2005} investigated the statistical consequences of such inequalities, particularly focusing on the Devroye inequality (bound on the variance) and its applications to processes arising from dynamical systems modeled by Young towers with exponential decay of correlations for Hölder observables. Further contributions by Chazottes and Gouëzel \cite{ChaGou:2012} introduced optimal concentration inequalities for dynamical systems modeled by Young towers, providing a comprehensive framework that encompasses systems with both exponential and polynomial decay of correlations for Hölder observables.

These developments have not only deepened our understanding of the probabilistic aspects of dynamical systems but have also facilitated the application of concentration inequalities to a broader class of systems, including those with indifferent fixed points and slowly mixing behavior. The interplay between dynamical systems theory and concentration inequalities remains a fertile ground for research, with ongoing efforts to refine these inequalities and extend their applicability. 

We investigate concentration inequalities for random dynamical systems, establishing general bounds for separately Lipschitz observables. Most of the applications presented in this work concern an important subclass of observables that can be expressed as partial or time-averaged Birkhoff sums of a Lipschitz function along random orbits. In particular, in Section~\ref{sec:lawofLN} we establish a law of large numbers for Birkhoff sums, which can be viewed as a direct consequence of deviation estimates for these sums. These deviation estimates themselves follow from the concentration inequalities provided by Theorem~\ref{thm:MAIN-0}. Moreover, in Section~\ref{sec:aplicaciones}, we analyze synchronization phenomena at the level of time averages, convergence of empirical measures, and fluctuation bounds for estimators of the correlation dimension. 

Furthermore, we analyze two classes of random dynamical systems in which the abstract hypotheses are satisfied: finitely generated systems on the circle and projective systems induced by linear cocycles. In both cases, we obtain uniform exponential concentration inequalities for observables of dynamical relevance, such as Lyapunov exponents. These examples highlight how structural properties—such as proximality, local contraction, or cocycle regularity—lead to stochastic stability and sharp concentration phenomena for observables of dynamical interest.

The article is organized as follows. In Section~\ref{sec:defRDS}, we review the formal framework and introduce the basic definitions. Our main result is stated in Section~\ref{sec:MR-0}. It is then established and proved in a slightly more general setting in Section~\ref{subsub:mainresults}, where we also present an alternative version of the theorem together with several direct consequences. The proofs of the main results are given in Section~\ref{sec:proofmainr}. In Section~\ref{sec:WeakCA}, we investigate structural conditions ensuring weak contraction on average, including a proof of an almost-sure central limit theorem. Section~\ref{sec:aplicaciones} is devoted to applications of concentration inequalities to systems satisfying the weak contraction on average condition. Finally, in Section~\ref{sec:examples}, we provide examples illustrating the applicability of our main results.

\subsection{Random dynamical systems}\label{sec:defRDS}

Given a probability measure $\nu$ on a suitable $\sigma$-algebra over the space of continuous maps from a fixed topological space into itself, one obtains a \emph{random dynamical system (RDS)} by selecting independently, at each integer time $n$, a map $F_n$ according to the law $\nu$ and applying it to the current state.

Random compositions of maps have a long history, going back at least to the study of products of random matrices and linear cocycles, notably in the seminal work of Furstenberg and Kesten \cite{FurKes:1960}.
In a nonlinear setting, Hutchinson \cite{hut:1981} introduced the framework of \emph{iterated function systems}, corresponding to the case where $\nu$ is supported on a finite family of contractions on a complete metric space, and established the existence of invariant measures and attractors in this context.
A general and systematic theory of random dynamical systems was later developed by Kifer \cite{Kifer:86}, among others.

Before introducing the formal setting, let us first, describe the spaces that will be used throughout this work.
Let $(M,d)$ be a metric space, and let $C(M)$ denote the space of continuous maps $f \colon M \to M$. We consider a metric subspace $(C_\varrho, \varrho)$ of $C(M)$, adapted to the regularity of the functions under consideration. For instance, \( C_\varrho \) may be the full space \( C(M) \) equipped with the supremum metric; the space of Lipschitz maps with the standard Lipschitz metric; or the space of differentiable maps of class \( C^r \) (e.g., \( C^1 \), \( C^{1+\tau} \), or \( C^2 \)), endowed with the corresponding \( C^r \)-type metric.
 This flexible framework allows for a unified and rigorous treatment of the dynamical and probabilistic properties of a broad range of function classes.

Consider a Borel probability measure $\nu\in\mathrm{Prob}\left(C_\varrho \right)$
with topological support $\cF\coloneqq\supp\,\nu$.
Define the product probability space $(\Omega,\,\bP) \coloneqq\left(\cF^{\mathbb{N}},\,\nu^{\mathbb{N}}\right)$\footnote{The left shift map
$\theta:\,\Omega\to\Omega$, given by $
\theta(f_1,f_2,f_3,\dots)=(f_2,f_3,f_4,\dots)$, is measure-preserving and ergodic with respect to $\bP$.} equipped with its product sigma-algebra. We denote by $\bE$ the expectation with respect to the probability measure $\bP$. Also, consider the \emph{coordinate process}
$(F_n)_{n\in\bN}$
given by 
\begin{equation}\label{eq:defcoord}
   F_n[\omega]\coloneqq f_n,\quad n\in\bN,\quad \omega=(f_1,f_2,\dots)\in\Omega.
\end{equation}
The sequence $(F_n)_{n\in\bN}$ is i.i.d., with each $F_n$ being a $\cF$-valued random variable with common distribution $\nu$.

We consider the \emph{Random Dynamical System (RDS)} associated with $\nu$, defined by the cocycle\footnote{The map $T:(n,\omega)\mapsto G_n[\omega]$ satisfies the cocycle property $T(n+m,\omega)=T(n,\theta^m(\omega))\circ T(m,\omega)$.} $(n,\omega,x)\mapsto G_n[\omega](x),$ where $(G_n)_{n\in\bN}$ is the \emph{random walk} on the semigroup generated by $\cF$ given by 
\begin{equation}\label{eq:defGn}
    G_n[\omega] \coloneqq f_{n} \circ \cdots \circ f_1,\quad n\in\bN, \quad \omega=(f_1,f_2,\dots)\in\Omega,
\end{equation}
with the convention $G_0=\operatorname{id}_M$. 
To simplify the notation, the dependence on the random parameter $\omega$ will often be omitted.
Thus, quantities such as $G_n(x)$ are understood as random variables on $(\Omega,\mathbb{P})$, with $\omega$ kept implicit.
This convention will be used consistently throughout the paper.

Let $\eta \in \mathrm{Prob}(M)$ be a Borel probability measure, and let $X_0$ be an $M$-valued random variable on $(\Omega,\mathbb{P})$ with law $\eta$, that is,
\[
\mathbb{P}(X_0 \in A) = \eta(A)
\quad \text{for all Borel sets } A \subset M.
\]
We define the \emph{fiber Markov chain} $(X_n)_{n\geq 0}$ by
\begin{equation}\label{eq:def-processXn}
   X_n \coloneqq G_n(X_0), \quad n\geq 0.
\end{equation}
We also define the \emph{skew Markov chain} $(Y_n)_{n\geq 0}$ by
\begin{equation}\label{eq:def-processYn}
    Y_n \coloneqq \left(F_{n+1}, X_n\right), \quad n\geq 0.
\end{equation}
In the particular case where $X_0 \equiv x$ $\mathbb{P}$-almost surely, we denote the corresponding processes by $(X_n^x)$ and $(Y_n^x)$.
Throughout the paper, we use the notation $X_0^x$ in place of $x$ when the initial condition is deterministic, in order to preserve a consistent recursive notation for the fiber dynamics.

\subsubsection{Notations}
Consider a general metric space $(\mathcal{X},\mathrm{dist})$. For a subset $A\subset \mathcal{X}$, define the \emph{diameter} of $A$ as follows
\[
|A|_\mathrm{dist}\coloneqq\sup_{\mathtt{x},\mathtt{y}\in A}\mathrm{dist}(\mathtt{x},\mathtt{y}).
\]
In this work, the metric space $\mathcal{X}$ is either the fiber space $M$, the base space $C_\varrho $, or the product space $\cF \times M$, endowed respectively with the metrics $d$, $\varrho $, or $\varrho +d$.

Fix $n\in\bN$. A function $\varphi\colon \mathcal{X}^{n+1} \to \mathbb{R}$ is said to be \emph{separately Lipschitz} if there exist nonnegative constants $\gamma_0,\gamma_1,\dots,\gamma_{n}$ such that for all  $\left.\mathtt{x}\right\vert_0^n$, $\left.\mathtt{y}\right\vert_0^n\in\mathcal{X}^{n+1}$
one has
\begin{equation}\label{eq:def:sep-lips-skew}
    \left|\varphi\left(\left.\mathtt{x}\right\vert_0^n\right)- \varphi\left(\left.\mathtt{y}\right\vert_0^n\right)\right|\leq \sum_{i=0}^{n}\gamma_i\,\mathrm{dist}(x_i,y_i)
\end{equation}
where the notation $\left.\mathtt{x}\right\vert_0^n=(\mathtt{x}_0,\ldots,\mathtt{x}_n)$ is used. The set of all functions satisfying \eqref{eq:def:sep-lips-skew} is denoted by $\Lip_{\mathrm{dist}}\left(\mathcal{X}^{n+1},\left.\gamma\right\vert_0^{n}\right)$.

We consistently use the notation $\left.\mathtt{x}\right\vert_m^n = (\mathtt{x}_m, \ldots, \mathtt{x}_n)$ to denote finite sequences indexed from $m$ to $n$. This convention applies in various contexts; for instance, $x\vert_m^n = (x_m, \ldots, x_n)$, $f\vert_m^n = (f_m, \ldots, f_n)$, $(f,x)\vert_m^n = ((f_m,x_m), \ldots, (f_n,x_n))$, $X^x\vert_m^n = (X_m^x, \ldots, X_n^x)$, and $Y^x\vert_m^n = (Y_m^x, \ldots, Y_n^x)$.

Given two functions $f,g:\bN\to\bR$, we write $f(n)\approx g(n)$ to mean that there exist constants $c_1,c_2>0$ such that
\begin{equation}\label{eq:def-approx}
    c_1 g(n)\leq f(n)\leq c_2 g(n)
\end{equation}
for all $n\in\bN$. The constants may depend on fixed parameters (such as an initial condition), but not on $n$.

\subsection{Main Result}\label{sec:MR-0}

We state a finite-time concentration inequality for observables evaluated along random trajectories, which constitutes the main probabilistic estimate of the paper. The result is presented here in a simplified form; the complete and technical statement, together with its proof, is given in Section~\ref{subsub:mainresults}. Throughout the following result, we work under the standing assumptions and notation introduced in Section~\ref{sec:defRDS}. Although the constant $\beta_n$ appearing in the statement is defined abstractly, it will be explicitly controlled in all applications discussed below. We also emphasize that no completeness assumption is imposed on the metric space $(M,d)$.

\begin{theoremA}\label{thm:MAIN-0}
Fix $ n \in \mathbb{N} $.
Let $ \gamma_0, \gamma_1, \dots, \gamma_n $ be nonnegative real numbers, with at least one strictly positive, and set
    \[
    \lambda_n = \sup_{\substack{x,y \in M}} \sum_{k=0}^{n} \bE \left[ d(X_k^x, X_k^y) \right]\quad\text{and}\quad
\beta_n = \, n\left( |\cF|_\varrho + \lambda_n \right)\,\max_{i=0,\ldots,n}\gamma_i.
\]
Then, for every function $ \varphi \in \mathrm{Lip}_{d+\varrho}\left((\cF \times M)^{n+1}, \gamma\vert_0^n \right) $, for all $ x \in M$ and all $ t > 0 $, we have
\[
\bP\left( \varphi(\left.Y^x\right\vert_0^n) - \bE \left[\varphi(\left.Y^x\right\vert_0^n)\right] > t \right) \le \exp\left( -\frac{nt^2}{12 \beta_n^2} \right).
\]
\end{theoremA}
 
In order for the above inequality to be non-trivial, one must assume that $\lambda_n<\infty$. Clearly, this condition forces the metric space $(M,d)$ to be bounded. We emphasize, however, that apart from this implicit boundedness, no further assumptions on the metric space $(M,d)$ are required. We also do not require the sequence \( (\lambda_n)_{n\geq 1} \) to be uniformly bounded in \( n \). In particular, the above concentration inequality remains meaningful even when \( \lambda_n \to \infty \) as \( n \to \infty \), provided this growth is sufficiently slow so that \( n / \beta_n^2 \to \infty \). In this case, one still obtains a nontrivial concentration estimate, possibly with a non-exponential rate in \( n \), as illustrated in Example~\ref{example:sqrt-n}. Stronger assumptions—namely, uniform control of \( \lambda_n \)—will be imposed in Section~\ref{sec:aplicaciones} in order to derive exponential decay on the right-hand side of the inequality for classes of observables of particular interest. Moreover, suitable growth conditions on \( \beta_n \) may ensure summability of the upper bound, allowing for almost sure conclusions via, for instance, the Borel--Cantelli lemma.

It is important to stress that Theorem~\ref{thm:MAIN} does not rely on any contractivity assumption on the Markov chain $(Y_n)_{n\geq 0}$ itself. In the applications motivating this work, the chain $(Y_n)_{n\geq 0}$, describing the joint evolution of the driving sequence of maps and the state variable, is not contractive in any metric sense. On the contrary, the dynamics on the base space (or equivalently, on the space of maps) is usually expansive. In particular, RDSs generated by a finite family of maps, where the base dynamics is the shift on a space of symbol sequences, which is expansive from a dynamical viewpoint.

\section{Concentration inequalities}\label{subsub:mainresults}

Throughout this section, we assume that $(M,d)$ is a metric space, and $\nu$ is a Borel probability measure on a metric subspace $(C_\varrho, \varrho)$ of $C(M)$. We also assume that the topological support $\cF$ of $\nu$ is $\varrho$-bounded. Consider the associated Markov chains $(Y_n)_{n \geq 0}$ and $(X_n)_{n \geq 0}$, defined as in \eqref{eq:def-processYn} and \eqref{eq:def-processXn}, respectively.

Before stating the extended (and more abstract) version of Theorem~A, we pause to present an illustrative example of a random dynamical system.
This example shows that Theorem~\ref{thm:MAIN} yields nontrivial finite-time concentration estimates even in the absence of uniform or exponential contraction.
In this setting, the fiber dynamics exhibits only polynomial decay of distances and a vanishing Lyapunov exponent, which leads to a logarithmic growth of $\lambda_n$.
As a consequence, the resulting concentration rate is slower than in the uniformly contractive case, but it remains quantitative and fully explicit.
This example highlights the robustness of Theorem~\ref{thm:MAIN-0} beyond the classical Wasserstein-contracting framework.

\begin{example}\label{example:sqrt-n}
Consider a probability measure $\nu$ supported on the family
\[
\left\{ h_\alpha \colon [0,1]\to[0,1] \;\colon\; \alpha\geq 1\right\},
\qquad
\text{where}
\qquad
h_\alpha(x)=\frac{x}{1+\alpha x}.
\]
Let $\cF$ be the topological support of $\nu$. Considering $d$ as the usual metric on $M=[0,1]$ and $\varrho$ as the supremum distance on the family of the maps above, we can conclude that $|\cF|_\varrho\le \frac12$.

Let $(X_n^x)_{n\ge 0}$ be the fiber Markov chain associated with $\nu$ as in \eqref{eq:def-processXn}.
Let $(B_k)_{k\ge 1}$ be an i.i.d.\ sequence of random variables taking values in $[1,\infty)$ with law induced by $\nu$.
For every $x\in[0,1]$ and $n\ge 1$, the iterates admit the explicit representation
\[
X_n^x=\frac{x}{1+(B_1+\cdots+B_n)x}.
\]
In particular, $X_n^x\to 0$ almost surely for every $x\in[0,1]$, but the convergence is not exponential.
For all $n\ge 1$, we have
\[
\sup_{x,y\in[0,1]}\bE\!\left[d(X_n^x,X_n^y)\right]
\le \bE\!\left[X_n^1\right]
\le \frac{1}{n+1},
\]
so that
\[
\lambda_n
\le \sum_{k=1}^n \frac{1}{k+1}
\leq  \log(n+1).
\]
By Theorem~\ref{thm:MAIN-0}, for any $c>0$, all $n\in\bN$ and every observable
\[
\varphi \in \mathrm{Lip}_{d+\varrho}\!\left((\cF \times M)^{n+1}, \gamma_0^n \right)
\qquad\text{with}\qquad
\gamma_i\leq \frac{c}{n},
\]
we have, for all $x\in M=[0,1]$ and all $t>0$,
\[
\bP\left( \varphi(\left.Y^x\right\vert_0^n)
- \bE \left[\varphi(\left.Y^x\right\vert_0^n)\right] > t \right)
\le  \exp\left( -\frac{nt^2}{3 c^2(1+2\log n)^2} \right).
\]
Thus, although the concentration rate is weaker than exponential in $n$, it still decays to zero as $n\to\infty$.

If one considers instead the metric
\(
d_r(x,y)=|x-y|^r, \  r\in\left(\tfrac12,1\right)
\) on $M$,
then the above argument yields, for some constant $c>0$,
\[
\bP\left( \varphi(\left.Y^x\right\vert_0^n)
- \bE \left[\varphi(\left.Y^x\right\vert_0^n)\right] > t \right)
\le  c\,\exp\!\left( -c\, n^{\,1-2r} t^2 \right).
\]
Although this bound is not exponential in $n$, it is summable in $n$ and thus,
by the Borel--Cantelli lemma, implies almost sure convergence, even though
$\lambda_n$ diverges polynomially.
\end{example}

We now state a more general version of Theorem~\ref{thm:MAIN-0}.
Although the formulation may appear somewhat abstract at first glance, the essential difference with respect to the simpler setting lies in the definition of the quantity $\lambda_n$.
Instead of taking the supremum over all pairs of points in the entire space $M$, we allow this control to be localized on a finite family of subsets of $M$.

\begin{theorem}\label{thm:MAIN}
Fix $ n \in \mathbb{N} $. Let $ (M,d) $ be a metric space, and let $ \nu $ be a Borel probability measure on a metric subspace $ (C_\varrho, \varrho) $ of $ C(M) $. Assume that the topological support $ \cF $ of $\nu$ is $ \varrho $-bounded. 
Consider the associated Markov chains $ (Y_k)_{k\geq 0} $ and $ (X_k)_{k\geq 0} $, defined in \eqref{eq:def-processYn} and \eqref{eq:def-processXn}, respectively.
Assume there exist $ \ell \in \mathbb{N} $ and pairwise disjoint closed subsets $ \mathcal{I}_1, \dots, \mathcal{I}_\ell \subset M $ such that:
\begin{enumerate}
    \item For every $ i \in \{1, \ldots, \ell\} $ and every $ f \in \cF $, there exists $ j \in \{1, \ldots, \ell\} $ such that $ f(\mathcal{I}_i) \subset \mathcal{I}_j $;
    \item and the quantity
    \[
    \lambda_n = \sup_{i \in \{1, \ldots, \ell\}}\,\sup_{\substack{x,y \in \mathcal{I}_i}} \sum_{k=0}^{n} \bE \left[ d(X_k^x, X_k^y) \right]
    \]
    is finite.
\end{enumerate}

Let $ \gamma_0, \gamma_1, \dots, \gamma_n $ be nonnegative real numbers, with at least one strictly positive, and define
\[
\beta_n = \, n\left( |\cF|_\varrho + \lambda_n \right)\,\max_{i=0,\ldots,n}\gamma_i.
\]
Then, for every function $ \varphi \in \mathrm{Lip}_{d+\varrho}\left((\cF \times M)^{n+1}, \gamma\vert_0^n \right) $, for all $ x \in \bigcup_{i=1}^\ell \mathcal{I}_i $ and all $ t > 0 $, we have
\[
\bP\left( \varphi(\left.Y^x\right\vert_0^n) - \bE \left[\varphi(\left.Y^x\right\vert_0^n)\right] > t \right) \le \exp\left( -\frac{nt^2}{12 \beta_n^2} \right).
\]
\end{theorem}
\medskip

We use the same notation for $\lambda_n$ as in Theorem~\ref{thm:MAIN-0}.
When $\ell = 1$ and $\mathcal{I}_1 = M$, the definition of $\lambda_n$ in Theorem~\ref{thm:MAIN} coincides exactly with the one introduced in Theorem~\ref{thm:MAIN-0}.
In particular, Theorem~\ref{thm:MAIN-0} is a special case of Theorem~\ref{thm:MAIN}, corresponding to the choice $\ell = 1$ and $\mathcal{I}_1 = M$.

\medskip

The introduction of the sets \( \mathcal{I}_1,\dots,\mathcal{I}_\ell \) is primarily motivated by the need to include important examples that fall outside the scope of Theorem~\ref{thm:MAIN-0}, such as random dynamical systems on the circle, which are studied in greater detail in Section~\ref{sec:RDScircle}.
The closed sets \( \mathcal{I}_1,\dots,\mathcal{I}_\ell \) allow for a local control of the dynamics. Although no global boundedness assumption on \( M \) is imposed, we implicitly assume boundedness on each set \( \mathcal{I}_i \).

The following two results are immediate consequences of Theorem \ref{thm:MAIN}.
\begin{corollary}\label{cor:MAINfib}
Under the assumptions of Theorem \ref{thm:MAIN}, for every $\varphi\in \Lip_{d}\left( M^{n+1},\left.\gamma\right\vert_0^{n}\right)$, for all $x\in \cup_{i=1}^{\ell} \cI_i$, and for all $t>0$, we have
\[
\bP\left( \varphi(\left.X^x\right\vert_0^n)-\bE \left[\varphi(\left.X^x\right\vert_0^n)\right] > t\right)\le \exp\left(-\frac{nt^2}{12\beta_n^2}\right).
\]
\end{corollary}

Corollary~\ref{cor:MAINfib} applies to a broader class of random RDSs than those generating Markov chains that are contractive in the Wasserstein metric, which is equivalent to the notion of contraction on average introduced in Section~\ref{sec:CARDSs}. In particular, results of this type are already well established in the general framework of Wasserstein-contracting Markov chains; see, for instance, \cite[Section 23.4]{2018:DouMouPri}.

A notable feature of Corollary~\ref{cor:MAINfib} is that it does not require, \emph{a priori}, the existence of a stationary probability measure, in contrast with results such as \cite[Theorem~3.1]{2026:Wang}. For instance, the RDS on the interval $[0,1]$ in Example~\ref{exemplo:NCA} discussed in Section~\ref{sec:WeakCA} admits a unique stationary probability measure $\eta = \delta_0$.
Hence, one may consider the RDS restricted to the open interval $(0,1)$ and the resulting Markov chain admits no stationary probability measure. Even in this case, the RDS satisfies the hypotheses of Theorem~\ref{thm:MAIN}, and consequently the conclusion of Corollary~\ref{cor:MAINfib} applies to this restricted chain.


\begin{corollary}\label{cor:MAIN-base}
Under the assumptions of Theorem \ref{thm:MAIN}, for every $\varphi\in \Lip_{\varrho} \left( \cF^{n},\left.\gamma\right\vert_1^{n}\right)$, and for all $t>0$, we have
\[
\bP\left( \varphi(\left.F\right\vert_1^{n})-\bE \left[\varphi(\left.F\right\vert_1^{n})\right] > t\right) \le \exp\left(-\frac{nt^2}{12\beta_n^2}\right).
\]
\end{corollary}

Corollary~\ref{cor:MAIN-base} does not provide a new concentration inequality in itself, since inequalities of Hoeffding type for Lipschitz observables of i.i.d.\ sequences are classical and go back to the seminal work of Hoeffding~\cite{1963:Hoeffding}. In particular, when the driving sequence of maps is independent and identically distributed, concentration bounds of Gaussian type are well understood and can be obtained by standard probabilistic arguments.

Nevertheless, Corollary~\ref{cor:MAIN-base} is included in order to emphasize that the bound obtained here arises as a direct consequence of the dynamical structure induced by the random system on the space on which it acts. Although the resulting estimate is not necessarily sharper than classical i.i.d.\ bounds, it explicitly reflects the dependence on the dynamics through the quantities appearing in $\beta_n$. In this sense, the corollary illustrates how dynamical information can be incorporated into concentration estimates, even in situations where the underlying randomness is purely i.i.d.


In the applications developed in the subsequent sections, the functions to which these results are applied typically satisfy $ \gamma_i = c/n $ for some constant $ c > 0 $ and all $ i = 0, \ldots, n $. Moreover, under our assumption of weak contraction on average (see \eqref{eq:def-lambdanu}), the sequence $ (\lambda_n)_n $ remains bounded. Consequently, the quantity $ \beta_n $ is uniformly bounded in $ n $.

We now present a variation of Theorem \ref{thm:MAIN}, which may be more useful in situations where the coefficients $\gamma_i$ are not all simultaneously of the form $c/n$.

\begin{theorem}\label{thm:moulins-skew}
Assume the hypotheses of Theorem \ref{thm:MAIN}. For
$
k\in \{0,1,\ldots,n\},
$
define
\[
 u_k=\sup_{i\in\{1,\ldots,\ell\}}\sup_{\substack{x,y\in \cI_i}}
\bE \left[d(X_k^x,X_k^y)\right],\quad \alpha_k = \gamma_k|\cF|_\varrho +\sum_{j=1}^{n-k}\gamma_{k+j} u_{j-1},
\]
and $\alpha^2 = \sum_{k=0}^{n}\alpha_k^2$.
Then, for every $\varphi\in \Lip_{d+\varrho }\left((\cF\times M)^{n+1},\left.\gamma\right\vert_0^{n}\right)$ and for all $x\in \cup_{i=1}^{\ell} \cI_i$, and for all $t>0$, we have
\[
\bP\left( \varphi(\left.Y^x\right\vert_0^n)-\bE \left[\varphi(\left.Y^x\right\vert_0^n)\right] > t\right)\le \exp\left(-\frac{t^2}{12\alpha^2}\right).
\]
\end{theorem}

\subsection{Auxiliary results}

Let $n \in \mathbb{N}$ be fixed. Let $\cI_1, \dots, \cI_\ell \subset M$ be closed, pairwise disjoint subsets such that, for each $i \in \{1, \ldots, \ell\}$ and every $f \in \cF$, there exists some $j \in \{1, \ldots, \ell\}$ satisfying $f (\cI_i) \subset \cI_j;$ and
    \[
    \sup_{i \in \{1, \ldots, \ell\}} \sup_{\substack{x,y \in \cI_i}} \sum_{k=0}^n \bE \left[d(X_k^x, X_k^y)\right] < \infty.
    \]

Let $\gamma_0, \gamma_1, \dots, \gamma_n$ be fixed nonnegative real numbers, with at least one strictly positive. Additionally, let 
\[
\varphi \in \mathrm{Lip}_{d+\varrho}\left((\cF \times M)^{n+1}, \gamma_0^n\right)
\]
be fixed. 

Before proving the results stated in Section \ref{subsub:mainresults}, let us establish some key preliminary results.

Define the functions $g_k \colon (\cF\times M)^{k+1}\to\bR$ for $k\in{0,\ldots,n-1}$ by
\[
g_k(\left.(f,z)\right\vert_0^{k})\coloneqq\bE \left[\varphi\left(\left.(f,z)\right\vert_0^{k},\, \left.Y^{f_k(z_k)}\right\vert_0^{n-k-1} \right)\right].
\]
Consider also $g_n:(\cF\times M)^{n+1}\to\bR$ given by $g_n=\varphi$. Let us write the following decomposition
\begin{equation}\label{eq:01-teorem23.4.3}
g_{n}(\left.(f,z)\right\vert_0^{n}) = \sum_{k=1}^{n} \left[ g_{k}(\left.(f,z)\right\vert_0^k) - g_{k-1}(\left.(f,z)\right\vert_0^{k-1}) \right] + g_0(f_0,z_0), 
\end{equation}
and for $k\in\{0,\ldots,n-1\}$ and $ \left.(f,z)\right\vert_0^{k} \in M^{k+1} $,
\begin{align*}
g_{k}(\left.(f,z)\right\vert_0^{k}) 
&=\bE\left[g_{k+1}(\left.(f,z)\right\vert_0^{k},(F_1,f_{k}(z_{k}))\right]\\
&=\bE\left[g_{k+1}\left(\left.(f,z)\right\vert_0^{k},Y_0^{f_{k}(z_{k})}\right)\right].
\end{align*}

Let $u_k$ and $\alpha_k$ be as defined in Theorem \ref{thm:moulins-skew}.
Following the approach of \cite[Lemma 23.4.4]{2018:DouMouPri}, we now establish a key estimate that constitutes the core of the proofs of our main theorems:
\begin{lemma}\label{claim00001}
Assume the hypotheses of Theorem \ref{thm:MAIN}.
Let $k\in\{0,\ldots,n-1\}$ and $s>0$. For all $\left.f\right\vert_0^k\in\cF^{k+1}$, $\left.z\right\vert_0^{k-1}\in M^k$, and $z_k\in \cup_{i=1}^\ell \cI_i$,
we have
\[
\bE\left[\e^{s\,g_{k+1}(\left.(f,z)\right\vert_0^{k},(F_1,f_{k}(z_k)))}\right]
    \leq \e^{3\,s^2\,\alpha_{k+1}^2}\e^{s\,g_k(\left.(f,z)\right\vert_0^{k})}.
\]
\end{lemma} 

\begin{proof}
Note that the hypothesis $ u_0<\infty$ implies that each  $\cI_i$ is $d$-bounded.
Since $\cF$ is $\varrho$-bounded, $\varphi$ is a bounded continuous function. Hence, without loss of generality, we may assume that 
\begin{equation}\label{eq:varphipositive}
    \varphi \geq 0;
\end{equation} otherwise, we consider $(\varphi-\inf \varphi)$ instead of 
$\varphi$.

Fix $\left.f\right\vert_0^k\in\cF^{k+1}$, $\left.z\right\vert_0^{k-1}\in M^k$ and $z_k\in \cup_{i=1}^\ell \cI_i$. For $i\in\{1,\ldots,\ell\}$, set
\[
\cF_i=\{f\in\cF\,\colon\, f(f_k(z_k))\in \cI_i\}
\]
and
\[
T_ig_k(\left.(f,z)\right\vert_0^{k})=\bE\left[ g_{k+1}(\left.(f,z)\right\vert_0^{k},(F_1,f_{k}(z_k))\,\mathbbm{1}_{\cF_i}(F_1)\right].
\]
Note that 
\[
\sum_{i=1}^{\ell}T_ig_k(\left.(f,z)\right\vert_0^{k})=\bE\left[ g_{k+1}(\left.(f,z)\right\vert_0^{k},(F_1,f_{k}(z_k))\right]=g_{k}(\left.(f,z)\right\vert_0^{k}).
\]
Since $\varphi\geq 0$, we have 
\[
T_ig_k(\left.(f,z)\right\vert_0^{k})\leq g_{k}(\left.(f,z)\right\vert_0^{k}).
\]
For all $t \in [0,1]$ and $\hat f\in\cF_i$, we have
\begin{align}\label{eq:TH1-skew-bound}
    &(1-t)T_ig_k(\left.(f,z)\right\vert_0^{k}) + t g_{k+1}(\left.(f,z)\right\vert_0^{k},(\hat f,f_{k}(z_k))\nonumber\\
    &\leq T_ig_k(\left.(f,z)\right\vert_0^{k}) + \gamma_{k+1}\bE\left[ \varrho(F_1,\hat{f})\right]\nonumber\\
    &\quad\quad\quad\quad+\sum_{j=2}^{n-k}\gamma_{k+j}\bE \left[d(X_{j-2}^{\hat{f}(f_k(z_k))}, X_{j-2}^{F_1(f_k(z_k))})\,\mathbbm{1}_{\cF_i}(F_1)\right]\nonumber\\
    &\leq T_ig_k(\left.(f,z)\right\vert_0^{k}) +  \gamma_{k+1} |\cF|_\varrho +\sum_{j=2}^{n-k}\gamma_{k+j} u_{j-2}\\
    &= T_ig_k(\left.(f,z)\right\vert_0^{k}) +  \alpha_{k+1}\nonumber.
 \end{align}
 Fix $s>0$. For $i\in\{1,\ldots,\ell\}$ and $\hat f\in\cF_i$, set 
 \[ 
 \phi(t) = \exp \left((1-t)\,s\,T_ig_k(\left.(f,z)\right\vert_0^{k}) + t\,s\, g_{k+1}(\left.(f,z)\right\vert_0^{k},(\hat f,f_{k}(z_k))\right), \quad t \in [0,1]. 
 \]
Writing
\[
    \phi(1) \leq \phi(0) + \phi'(0) + \sup_{t \in [0,1]} \frac{\phi''(t)}{2}
\]
and integrating over $\cF_i$ yields
\begin{align*}
    &\int_{\cF_i}\e^{s\,g_{k+1}(\left.(f,z)\right\vert_0^{k},(\hat f,f_{k}(z_k))}\dd\nu(\hat f)\\
    &\leq \nu(\cF_i)\,\e^{s\,T_ig_k(\left.(f,z)\right\vert_0^{k})} + \frac{1}{2}\,\nu(\cF_i)\,s^2\,\alpha_{k+1}^2\,\e^{s\,T_ig_k(\left.(f,z)\right\vert_0^{k}) +  s\,\alpha_{k+1}} ,
\end{align*}
where we have used that
\[
\int_{\cF_i} \left[ g_{k+1}\left(\left.(f,z)\right\vert_0^{k},(\hat f,f_{k}(z_k))\right) - T_ig_k(\left.(f,z)\right\vert_0^{k})\right]^2 \dd\nu(\hat f)  \leq \nu(\cF_i)\,\alpha_{k+1}^2.
\]
Hence, using \eqref{eq:varphipositive}, we get
\begin{align*}
    &\int_{\cF_i}\e^{s\,g_{k+1}(\left.(f,z)\right\vert_0^{k},(\hat f,f_{k}(z_k))}\dd\nu(f)\\
    &\leq \nu(\cF_i)\,\e^{s\,g_k(\left.(f,z)\right\vert_0^{k})}\left(1 + \frac{1}{2}\,s^2\,\alpha_{k+1}^2\,\e^{s\,\alpha_{k+1}}\right) \,.
\end{align*}
Now, summing over $i\in\{1,\ldots,\ell\}$  
\begin{align*}
    \bE\left[\e^{s\,g_{k+1}(\left.(f,z)\right\vert_0^{k},(F_1,f_{k}(z_k))}\right]&=\int_{\cF}\e^{s\,g_{k+1}(\left.(f,z)\right\vert_0^{k},(\hat f,f_{k}(z_k))}\dd\nu(f)\\
    &\leq \,\e^{s\,g_k(\left.(f,z)\right\vert_0^{k})}\left(1 + \frac{1}{2}s^2\,\alpha_{k+1}^2\,\e^{s\alpha_{k+1}}\right) ,
\end{align*}
where we have used that $1=\nu(\cF)=\sum_{i=1}^{\ell}\nu(\cF_i)$.
Finally, using the elementary inequality
\[
1+\frac{u^2 e^u}{2}\le e^{3u^2}, \qquad u\ge 0,
\]
proved in Lemma~\ref{Appendix:lemma01}, we conclude the proof.
\end{proof}

 Proceeding exactly as in the proof of Lemma \ref{claim00001}, with the only difference being the use of different upper bounds in inequality \eqref{eq:TH1-skew-bound}, we can stablish:
\begin{lemma}\label{claim00001-01}
    Assume the hypotheses of Theorem \ref{thm:MAIN}. Let $k\in\{0,\ldots,n-1\}$ and $s>0$. For all $\left.f\right\vert_0^k\in\cF^{k+1}$, $\left.z\right\vert_0^{k-1}\in M^k$, and $z_k\in \cup_{i=1}^\ell \cI_i$,
we have 
\[
\bE\left[\e^{s\,g_{k+1}(\left.(f,z)\right\vert_0^{k},(F_1,f_{k}(z_k)))}\right]
    \leq \e^{3\,s^2(\beta_n/{n})^2}\e^{s\,g_k(\left.(f,z)\right\vert_0^{k})}.
\]
\end{lemma}

\begin{lemma}\label{claim00002}
Assume the hypotheses of Theorem \ref{thm:MAIN}. Let $s>0$. For all $z\in \cup_{i=1}^\ell \cI_i$,
we have 
\[
\bE\left[\e^{sg_0(F_1,z)}\right]
    \leq \e^{3\,s^2\,\alpha_0^2}\e^{s\bE\left[g_0(F_1,z)\right]}.
\]
\end{lemma} 

\begin{proof}
Without loss of generality, we assume that $\varphi \geq 0$.

Fix $z\in \cup_{i=1}^\ell \cI_i$. For $i\in\{1,\ldots,\ell\}$, set
\[
\cF_i=\{ f\in\cF\,\colon\,  f(z)\in \cI_i\}.
\]
For all $t \in [0,1]$ and $\hat f\in\cF_i$, we have
\begin{align*}
    &(1-t)\int_{\cF_i} g_0(f,z)\,\dd\nu(f)+ t g_{0}(\hat f,z)\nonumber\\
    &\leq \int_{\cF_i} g_0(f,z)\,\dd\nu(f)+\int_{\cF_i}\left[\gamma_0\varrho(f,\hat f)+\sum_{j=1}^n\gamma_j\bE \left[d\left(X_{j-1}^{f(z)},X_{j-1}^{\hat f(z)}\right)\right]\right]\dd\nu(f)\\
    &\leq \int_{\cF_i} g_0(f,z)\,\dd\nu(f)+\alpha_0\nonumber.
 \end{align*}
 For $i\in\{1,\ldots,\ell\}$ and $\hat f\in\cF_i$, define 
 \[ 
 \phi(t) = \exp \left\{(1-t)\int_{\cF_i} g_0(f,z)\,\dd\nu(f)+ t g_{0}(\hat f,z)\right\}, \quad t \in [0,1].
 \]
Follow the same reasoning as in the proof of Lemma \ref{claim00001} to conclude.
\end{proof}

Proceeding analogously to the proof of the previous lemma, we can prove:
\begin{lemma}\label{claim00002-01}
    Assume the hypotheses of Theorem \ref{thm:MAIN}. Let $k\in\{0,\ldots,n-1\}$ and $s>0$. For all $\left.f\right\vert_0^k\in\cF^{k+1}$, $\left.z\right\vert_0^{k-1}\in M^k$, and $z_k\in \cup_{i=1}^\ell \cI_i$,
we have 
\[
\bE\left[\e^{sg_0(F_1,z)}\right]
    \leq \e^{3\,s^2(\beta_n/{n})^2}\e^{s\bE\left[g_0(F_1,z)\right]}.
\]
\end{lemma}

\subsection{Proofs of the main concentration results}\label{sec:proofmainr}

We now proceed to the proofs of Theorems \ref{thm:MAIN} and \ref{thm:moulins-skew}.  
Let $x \in \bigcup_{i=1}^\ell \cI_i$ be fixed.  
Using the decomposition established in \eqref{eq:01-teorem23.4.3}, we write:
\begin{align}\label{eq:after claim-fos SSL}
\bE \left[\e^{s\,\varphi( Y^x\vert_0^n )}\right] &= \bE \left[\e^{s\,g_{n}( Y^x\vert_0^n )}\right]\nonumber \\
&= \bE \left[\e^{s\,\sum_{k=0}^{n-1} \left( g_{k+1}(Y_0^x,\ldots,Y_{k+1}^x) - g_{k}(Y_0^x,\ldots,Y_k^x) \right) + s\,g_0(Y_0^x)}\right] .
\end{align}

\begin{proof}[Proof of Theorem \ref{thm:MAIN}]
    By Lemma \ref{claim00001-01} and Lemma \ref{claim00002-01},
\begin{align*}
\bE \left[\e^{s\,\varphi( Y^x\vert_0^n )}\right]
&\leq \e^{3s^2\,(\beta_n^2/{n})}\,\e^{s\,\bE \left[\varphi( Y^x\vert_0^n )\right]}.
\end{align*}
Applying Markov’s inequality,
we obtain
\begin{equation*}
\bP\left( \varphi(Y^x\vert_0^n)-\bE \left[\varphi(Y^x\vert_0^n)\right] > t\right) \le \exp\left(-st+\frac{3s^2\beta_n^2}{n}\right).
\end{equation*}
Optimizing over $s>0$, the exponent is minimized at
$s = \frac{n t}{6\beta_n^2}$, which yields
\[
\bP\!\left( \varphi(Y^x\vert_0^n)-\bE[\varphi(Y^x\vert_0^n)] > t\right)
\le \exp\!\left(-\frac{n t^2}{12\beta_n^2}\right).
\]
This proves the theorem.
\end{proof}

\begin{proof}[Proof of Theorem \ref{thm:moulins-skew}]
Use \eqref{eq:after claim-fos SSL} and apply Lemmas \ref{claim00001} and \ref{claim00002},
\begin{align*}
\bE \left[\e^{s\,\varphi( Y^x\vert_0^n )}\right]
&\leq \e^{3s^2\,\sum_{k=0}^{n} \alpha_k^2 }\,\e^{s\,\bE \left[\varphi( Y^x\vert_0^n )\right]} =\e^{3s^2\,\alpha^2}\,\e^{s\,\bE \left[\varphi( Y^x\vert_0^n )\right]}.
\end{align*}
Applying Markov’s inequality,
we obtain
\begin{equation*}
\bP\left( \varphi( Y^x\vert_0^n )-\bE \left[\varphi( Y^x\vert_0^n )\right] > t\right) \le \exp\left(-st+3s^2\alpha^2\right).
\end{equation*}
Optimizing over $s>0$, the exponent is minimized at
$s = t/(6\alpha^2)$, which yields
\[
\bP\!\left( \varphi( Y^x\vert_0^n )-\bE[\varphi( Y^x\vert_0^n )] > t\right)
\le \exp\!\left(-\frac{t^2}{12\alpha^2}\right).
\]
This proves the theorem.
\end{proof}

\begin{proof}[Proof of Corollary \ref{cor:MAINfib}]
    Consider the function $\hat \varphi\colon(\cF\times M)^{n+1}\to\bR$ given by
    \[
     \hat \varphi(\left.(f,x)\right\vert_0^{n})=\varphi(\left.x\right\vert_0^{n}).
    \]
   One can verify that $\hat \varphi\in \Lip_{d+\varrho }\left((\cF\times M)^{n+1},\left.\gamma\right\vert_0^{n}\right)$. To conclude the proof, apply Theorem \ref{thm:MAIN} to the function $\hat \varphi$.
\end{proof}

\begin{proof}[Proof of Corollary \ref{cor:MAIN-base}]
    Consider the function $\hat \varphi\colon(\cF\times M)^{n+1}\to\bR$ given by
    \[
     \hat \varphi(\left.(f,x)\right\vert_0^{n})=\varphi(\left.f\right\vert_0^{n-1}).
    \]
   Observe that $\hat \varphi\in \Lip_{d+\varrho }\left((\cF\times M)^{n+1},\left.\gamma\right\vert_0^{n}\right)$. To conclude the proof, apply Theorem \ref{thm:MAIN} to the function $\hat \varphi$.
\end{proof}
\section{Weakly contracting on average RDSs}\label{sec:WeakCA}

Before presenting the applications of our concentration inequalities, we analyze a key structural condition that plays a central role throughout this work: weak contraction on average. This notion captures the idea that, on average, random trajectories tend to come closer over time, even if individual maps may not be contractions.

Throughout this section, we assume that $(M,d)$ is a compact (complete and bounded) metric space. Also, we assume that $\nu$ is a probability measure on $C(M)$ with topological support $\cF$ being $\varrho _\infty$-bounded, that is,
\begin{equation}\label{eq:def de Dnu}   |\cF|_\infty\coloneqq\sup_{f,g\in\cF}\varrho_\infty(f,g)<\infty,
    \end{equation}
where
\[
\varrho_\infty(f,g)=\sup_{x\in M}\,d(f(x),g(x)).
\]

We say that the RDS induced by $\nu$ is \emph{weakly contracting on average} on $M$ if 
\begin{equation}\label{eq:def-lambdanu}
\lambda_\nu\coloneqq\sup_{x,y\in M}\sum_{n=0}^{\infty}\bE \left[d(X_n^x,X_n^y)\right]<\infty.
\end{equation}
A large class of examples of RDSs exhibiting weak contraction is studied in \cite{2024:GelSal}. See \cite[Theorem 1.4]{2024:GelSal} for sufficient conditions under which 
$\lambda_\nu<\infty$ holds. See also \cite[Section 2]{2024:GelSal} for examples of RDSs on the projective space of $\bR^m$, $d\geq 2$, satisfying the conditions of \cite[Theorem 1.4]{2024:GelSal} and therefore the weakly contracting on average condition.

We say that the RDS induced by $\nu$ is \emph{uniformly weakly contracting on average} if 
\begin{equation}\label{eq:def-LLambdanu}
\sum_{n=0}^{\infty}\sup_{x,y\in M}\bE \left[d(X_n^x,X_n^y)\right]<\infty.
\end{equation}
It is clear that \eqref{eq:def-LLambdanu} implies \eqref{eq:def-lambdanu}. The family of RDSs that satisfy \eqref{eq:def-LLambdanu} includes, for example, all those that are contractive on average; see Section~\ref{sec:CARDSs}.

Note that we can write $\lambda_\nu$ in \eqref{eq:def-lambdanu} alternatively as follows
\[
\lambda_\nu=\sup_{N\geq 0}\sup_{x,y\in M}\sum_{n=0}^N\bE \left[d(X_n^x,X_n^y)\right].
\]

Let us show an example of RDS on the interval $[0,1]$ for which $\lambda_\nu<\infty$, but the sequence
\[
\left(\sup_{x,y\in M}\bE \left[d(X_n^x,X_n^y)\right]\right)_{n\in\bN}
\]
does not decay exponentially as $n\to\infty$.
\begin{example}\label{exemplo:NCA}
    Consider a probability measure $\nu$ on the family 
    \[
    \left\{h_\alpha\colon[0,1]\to[0,1]\quad\colon \alpha\in\left[\frac54,\frac32\right]\right\},
    \quad
    \text{where}\quad
    h_\alpha(x)=x-x^{\alpha}.
    \]
    Here the notation $\approx$ is used in the sense of \eqref{eq:def-approx}.
Note that for all $x\in(0,1]$ and $n\in\bN$, we have
    \[
    h_\alpha(x)< x,\quad h_\alpha^n(x)\leq h_\alpha^n\left(\alpha^{-\frac{1}{\alpha-1}}\right)\quad\text{and}\quad h_\alpha^n(x)\approx \frac{1}{(\alpha-1)n^{\frac{1}{\alpha-1}}}.
    \]
    Where the above approximation depends only on  $x\in(0,1]$.
    Consider $\cF=\supp \nu$. Let $(X_n^x)_{n\geq0}$ be the fiber Markov chain associated to $\nu$ as in \eqref{eq:def-processXn}. Then, almost surely for all $n\geq 1$ and each $x,y$ we have
    \[
    h_{\frac54}^n(x)\leq X_n^x\leq h_{\frac32}^n(x)\leq h_{\frac32}^n\left(\frac49\right),
    \]
    and so
    \[
    d(X_n^x,X_n^y)=\vert X_n^x-X_n^y\vert\leq h_{\frac32}^n\left(\frac49\right)\approx n^{-2}.
    \]
    Here, the above approximation depends only on $\frac32$, which is independent of $x$, $y$, and the variables $X_n^x,X_n^y$.
    Therefore, for constant $c>0$,
    \[
    \sup_{x,y\in[0,1]}\sum_{n\geq 0}\bE\left[d(X_n^x,X_n^y)\right]\leq 1+\sum_{n\geq 1}\frac{c}{n^2} <\infty.
    \]
    On the other hand, fix $a\in(0,1]$, then, 
    \[
    d(X_n^0,X_n^a)=X_n^a\geq h_{\frac54}^n(a)\approx \frac{4}{n^4}. 
    \]
    Hence, for all $n\geq 1$
    \[
    \sup_{x,y\in[0,1]}\bE \left[d(X_n^x,X_n^y)\right]\geq \bE \left[d(X_n^0,X_n^a)\right] \approx\frac{4}{n^4}.
    \]
    Which shows that the decay rate of $\sup_{x,y\in[0,1]}\bE\left[ d(X_n^x,X_n^y)\right]$ is polynomial in $n$.
\end{example}

\subsection{Key properties}\label{sec:WCA-KeyProp}

We recall that a probability measure $\eta \in \prob(M)$ is \emph{$\nu$-stationary} when the fiber Markov chain with initial distribution $\eta$ (i.e., $X_0$ is $\eta$-distributed $\bP(X_0\in A)=\eta(A)$) is a stationary process. Equivalently, $\eta \in \prob(M)$ is $\nu$-stationary, if
\begin{equation}\label{eq-def:stationary measure}
\eta(A)=\int_\cF \eta(f^{-1}A)\,\dd\nu(f)=\bP(X_1\in A),
\end{equation}
for all borelian set $A\subset M$. Indeed, stationarity of the process 
$(X_n)_{n\geq0}$ means that $X_0$ and $X_1$ (and hence all $X_n$) have the same law, which is exactly expressed by \eqref{eq-def:stationary measure}.

For the following result we do not need $\cF$ to be bounded.
\begin{proposition}\label{PROP:keypropweakcontract}
    Let $(M,d)$ be a compact metric space, and let $\nu$ be a probability measure on $C(M)$. Consider the fiber Markov chain  $(X_n)_{n\geq 0}$ associated to $\nu$ as in \eqref{eq:def-processXn}. Assume the RDS induced by $\nu$ is weakly contracting on average. Then, for all $x,y\in M$, $\bP$-almost surely
    \[
    \lim_{n\to\infty}d(X_n^x,X_n^y)=0.
    \]
    Moreover, there exists a unique $\nu$-stationary probability measure on $M$.
\end{proposition}
\begin{proof}
Fix $x,y\in M$. By the weak contraction on average assumption,
\[
\sum_{n=0}^{\infty}\bE\big[d(X_n^x,X_n^y)\big]<\infty.
\]
Thus, for any $\varepsilon>0$, Chebyshev's inequality yields
\[
\sum_{n=0}^{\infty}\bP\big(d(X_n^x,X_n^y)\ge\varepsilon\big)
\le \frac{1}{\varepsilon}\sum_{n=0}^{\infty}\bE\big[d(X_n^x,X_n^y)\big]<\infty.
\]
By the Borel--Cantelli lemma, $d(X_n^x,X_n^y)\to 0$ almost surely, proving the first statement.

The second statement follows from the first one and \cite[Proposition~1]{Sten:2012}.
\end{proof}

We say that the RDS induced by $\nu$ has \emph{decay of correlations for Lipschitz functions} with respect to a $\nu$-stationary measure $\eta$ on $M$ if there exist constants $c > 0$ and a summable sequence $(p_j)_{j \geq 1}$ (i.e., $\sum_j p_j < \infty$)  such that for all Lipschitz functions $g,h : M \to \bR$ we have, for all $j \geq 1$
\[
\left| \int_M h(x) \, \bE \left[g(X_j^x)\right] \, \dd\eta(x) - \int_M h \, \dd\eta \int_M g \, \dd\eta \right| 
\leq c\, p_j\, \|h\| \,\|g\|,
\]
where $\|h\|\coloneqq \|h\|_\infty + \operatorname{L}(h) $, and
\[
\operatorname{L}(h) \coloneqq\sup_{x \neq x'} \frac{|h(x)-h(x')|}{d(x,x')}<\infty\quad\text{and}\quad\|h\|_\infty\coloneqq\sup_{x}|h(x)|.
\]

\begin{proposition}[Decay of correlations]\label{PROP:keypropweakcont-deccorr}
    Let $(M,d)$ be a compact metric space.  
    Let $\nu$ be a probability measure on $C(M)$. If the RDS induced by $\nu$ is weakly contracting on average, then it exhibits decay of correlations for Lipschitz observables with respect to its stationary measure.
\end{proposition}
\begin{proof}
    Assume that the RDS induced by $\nu$ is weakly contracting on average. Let $\eta\in\prob(M)$ be the $\nu$-stationary measure. Set
    \begin{equation}\label{eq:defpj}
    p_j\coloneqq\int_M\int_M\bE \left[d(X_j^x,X_j^y)\right]\dd\eta(x)\dd\eta(y).
    \end{equation}
    Note that
    \[
    \sum_{j\geq 1}p_j\leq \lambda_\nu<\infty,
    \]
    for $\lambda_\nu$ as in \eqref{eq:def-lambdanu}. 

    Now, consider two Lipschitz functions $g,h : M \to \bR$. Since $\eta$ is $\nu$-stationary
\begin{align}\label{eq:prop-dec-corr}
&\left| \int_M h(x) \, \bE \left[g(X_j^x)\right] \, \dd\eta(x) - \int_M h \, \dd\eta \int_M g \, \dd\eta \right|\nonumber\\
&=\left| \int_M h(x)\left[  \bE \left[g(X_j^x)\right] -  \int_M \bE \left[g(X_j^y)\right] \, \dd\eta(y) \right]\dd\eta(x)\right|\nonumber\\
&\leq \|h\|_\infty\, \operatorname{L}(g)\int_M\int_M\bE \left[d(X_j^x,X_j^y)\right]\dd\eta(x)\dd\eta(y)\\
&\leq  p_j\,\|h\|\, \|g\|.\nonumber
\end{align}
The proposition is proved.
\end{proof}

The following result provides a moment concentration bound of order two.
It can be viewed as a Devroye inequality adapted to weakly contracting
on average random dynamical systems.
 
\begin{proposition}[Devroye inequality]\label{Prop:ConcBound2}
   Let $(M,d)$ be a compact metric space, and let $\nu \in \mathcal{P}(C(M))$ induce an RDS that is weakly contracting on average. Let $\eta \in \mathcal{P}(M)$ denote the corresponding stationary measure, and let $(X_n)_{n \geq 0}$ be the fiber Markov chain associated to $\nu$, as defined in \eqref{eq:def-processXn}. Fix $n \in \mathbb{N}$, and let $\lambda_\nu$ be as in \eqref{eq:def-lambdanu}.

Then, for every sequence $\gamma\vert_0^{n-1} \in (0,\infty)^n$ with $\gamma_k \geq \gamma_{k+1}$ for all $k$, and every function $\varphi \in \Lip_d(M^n, \gamma\vert_0^{n-1})$, we have
    \[
    \int_M\bE\left[\varphi(X^x\vert_0^{n-1})-\int_M\bE\left[\varphi(X^y\vert_0^{n-1})\right]\dd\eta(y)\right]^2\dd\eta(x)\leq \lambda_\nu\vert M\vert_d\sum_{k=0}^{n-1}\gamma_k^2.
    \]
\end{proposition}

\begin{proof}
For any $y\in M$, let $(\hat X^y_0,\ldots,\hat X^y_{n-1})$ be a random vector
defined on an auxiliary probability space $(\hat\Omega,\hat{\cF},\hat{\mathbb{P}})$,
with the same law as $(X^y_0,\ldots,X^y_{n-1})$, and independent of
$(X^x_0,\ldots,X^x_{n-1})$ for every $x\in M$.
We denote by $\hat{\mathbb E}$ the expectation with respect to $\hat{\mathbb P}$. Hence, for $x\in M$
\begin{align*}\label{eq:desc-SL}
    &\varphi(X^x\vert_0^{n-1})-\int_M\bE\left[\varphi(X^y\vert_0^{n-1})\right]\dd\eta(y)\nonumber\\
    &=\varphi(X^x\vert_0^{n-1})-\int_M\hat\bE\left[\varphi(\hat X^y\vert_0^{n-1})\right]\dd\eta(y)= \sum_{k=0}^{n-1}J_k^x, 
\end{align*}
where
\begin{align*}
J_0^x&=\varphi(X^x\vert_0^{n-1})-\int_M\hat\bE\left[\varphi(X^x\vert_0^{n-2},\hat X^y_{n-1})\right]\dd\eta(y),\\
J_k^x&=\int_M\hat\bE \left[\varphi(X^x\vert_0^{k},\hat X^y\vert_{k+1}^{n-1})-\varphi(X^x\vert_0^{k-1},\hat X^y\vert_{k}^{n-1})\right]\dd\eta(y),\\
 J_{n-1}^x&=\int_M\hat\bE\left[\varphi(X^x_0,\hat X^y\vert_1^{n-1})-\varphi(\hat X^y\vert_0^{n-1})\right]\dd\eta(y),
\end{align*}
for $k=1,\ldots,n-2$. The vectors
\[
(X^x|_0^{k},\hat X^y|_{k+1}^{n-1})
\quad\text{and}\quad
(X^x|_0^{k-1},\hat X^y|_{k}^{n-1})
\]
differ only at the $k$-th entry. Hence,
\begin{equation}\label{eq:desigforJk}
    \vert J_{k}^x\vert\leq \gamma_k \int_M d(X_k^x,y)\dd\eta(y).
\end{equation}

Note that $\int_M\bE \left[J_k^x\right]\dd\eta(x)=0$. Further, we have
\begin{align}\label{eq:0-var2}
&\int_M\bE\left[\varphi(X^x\vert_0^{n-1})-\int_M\bE\left[\varphi(X^y\vert_0^{n-1})\right]\dd\eta(y)\right]^2\dd\eta(x)\nonumber\\
&=\sum_{k=0}^{n-1}\int_M\bE\left[J_k^x\right]^2\dd\eta(x)+2\sum_{k=0}^{n-2}\sum_{j=k+1}^{n-1}\int_M\bE\left[J_k^x J_j^x\right]\dd\eta(x)
\end{align}
Using \eqref{eq:desigforJk} and 
\[
\left(\int_M d(X_k^x,y)\dd\eta(y) \right)^2\leq \vert M\vert_d \int_M d(X_k^x,y)\dd\eta(y),\]
we get
\begin{equation}\label{eq:1-var2}
\int_M\bE\left[J_k^x\right]^2\dd\eta(x)\leq \gamma_k^2\vert M\vert_d \int_M\int_M d(x,y)\dd\eta(y)\dd\eta(x).
\end{equation}
If $k<j$, then we can use the $\nu$-stationarity of $\eta$ to obtain
\begin{align*}
&\int_M\bE\left[J_k^x J_j^x\right]\dd\eta(x)\\
&\leq \gamma_k\gamma_j\int_M\bE\left[\int_M \int_M d(X_k^x,y)d(X_j^x,z)\dd\eta(y)\dd\eta(z)\right]\dd\eta(x)\\
&=\gamma_k\gamma_j\int_M\left(\int_M d(x,y)\dd\eta(y)\bE\left[\int_M d(X_{j-k}^x,z)\dd\eta(z)\right]\right)\dd\eta(x).
\end{align*}
Applying \eqref{eq:prop-dec-corr} with $h=g=\int_M d(\cdot,z)\dd\eta(z)$, we get
\begin{align}\label{eq:2-var2}
\int_M\bE\left[J_k^x J_j^x\right]\dd\eta(x)
\leq \gamma_k^2 \vert M\vert_d\, p_{j-k},
\end{align}
where $p_j$ is defined in \eqref{eq:defpj}. By combining \eqref{eq:0-var2}, \eqref{eq:1-var2} and \eqref{eq:2-var2}, the proposition is proved.
\end{proof}

\subsection{Almost-sure central limit theorem}\label{sec:AS-CLT}
The classical central limit theorem guarantees convergence in distribution, but
it does not provide almost sure information along typical realizations.
This limitation has motivated the development of \emph{almost sure central limit
theorems}, in which suitably weighted averages of partial sums converge almost
surely; see, for instance,~\cite{LaceyPhilipp:1990}.
In the context of dynamical systems, analogous results have been obtained by
considering logarithmically weighted averages of empirical measures; see~\cite{ChazottesGouezel:2007}.

In this section, we establish an almost sure central limit theorem for random
dynamical systems that are weakly contracting on average.
Although the concentration inequalities proved in
Section~\ref{subsub:mainresults} are not directly used here, the present result is
of independent interest, as it provides a precise description of the statistical
behavior of random orbits generated by such systems.

The proof relies crucially on the structural properties of weak average
contraction developed in Section~\ref{sec:WCA-KeyProp}.
In particular, we repeatedly use the inequality established in
Proposition~\ref{Prop:ConcBound2}, which plays a role analogous to Devroye’s
inequality in the proof of~\cite[Theorem~8.1]{ChazCollSchm:2005}.

Although we do not apply the concentration inequalities established in Section \ref{subsub:mainresults}, this subsection presents an interesting result: \emph{an almost sure central limit theorem}. This result is especially significant because it characterizes the statistical behavior of random orbits associated with an RDS that is weakly contracting on average. The proof uses the key properties for weakly average contraction developed in Section \ref{sec:WCA-KeyProp}. In fact, we repeatedly use the moment concentration bound of order 2. Such a result plays the role of Devroye's inequality, in the proof of \cite[Theorem 8.1]{ColMarSch:2002}.

For $x\in M$, $n\in\bN$ and a function $h\colon M \to\bR$ consider the random variable
\begin{align}\label{eq-def:Snx}
S_n^x(h)\coloneqq \sum_{k=0}^{n-1}h(X_k^x).
\end{align}

Applying the result in \cite{DerrLin:2003} and following the proof of the central limit theorem (CLT) in \cite{2024:GelSal}, we can establish the following CLT for each fiber chain $(X_n^x)_{n\in\bN}$, $x\in M$.

\begin{proposition}\label{prop:CLT}
  Let $(M,d)$ be a compact metric space. Let $\nu\in\prob(C(M))$ with its support $\cF$ being $\varrho_\infty$-bounded. Assume that the RDS induced by $\nu$ is weakly contracting on average. For $h\in\Lip_d(M)$ and $x\in M$, let $(S_n^x(h))_{n\geq 0}$ be defined as in \eqref{eq-def:Snx}.
 Then for any $h\in\Lip_d(M)$, with $\eta(h)=0$, the limit
\begin{equation}\label{eq:def-sigma2h}
\sigma^2_h\coloneqq\lim_{n\to\infty} \frac1n\int_M \bE[S_n^x(h)^2]\dd\eta(x)
\end{equation}
exists and is finite, and for every point $x\in M$  
\begin{equation}\label{eq:clt01}
\frac{1}{\sqrt{n}}S_n^x(h)\xrightarrow{\text { law }}\cN(0,\sigma^2_h)
\end{equation}
where ``law" stands for the convergence in law, and $\cN(0,\sigma^2_h)$ denotes the Gaussian distribution (if $\sigma^2_h=0$, it is the Dirac measure at 0).
\end{proposition}

\begin{remark}
    If $\sigma^2_h>0$, then \eqref{eq:clt01} is equivalent to
    \begin{equation}
        \lim_{n\to\infty}\bP\left( \frac{S_{n}^x(h)}{ \sqrt{n}} \leqslant t\right)=\frac{1}{\sigma_h\sqrt{2 \pi}} \int_{-\infty}^{t} \mathrm{e}^{-\frac{u^{2}}{2\sigma_h^2}} \mathrm{~d} u .
    \end{equation}
\end{remark}

For $\sigma> 0$, we denote by $\rho_\sigma\in\prob({\bR})$ the Gaussian measure on $\mathbb{R}$ defined by
\[
\rho_\sigma(B)\coloneqq\frac{1}{\sigma\sqrt{2 \pi}} \int_{B} \mathrm{e}^{-\frac{u^{2}}{2\sigma^2}} \,\dd\mathrm{Leb}(u), 
\]
for all Borelian set of $B\subset\bR$. Here, Leb is the Lebesgue measure. Usually, one simply writes $\dd u$ instead of $\mathrm{dLeb}(u).$ We adopt the convention that $\rho_0$ is the Dirac measure sitting at $0$.

\begin{remark}
We do not recall the definition of the Kantorovich distance here, since it is not
needed in this section; we only use the characterization below.
The precise definition can be found in~\eqref{def:dist-KANT}.
Let $(\mu_n)_{n\in\bN}$ and $\mu$ be probability measures on $\bR$. Then
\[
\lim_{n\to\infty} \kappa(\mu_n,\mu)=0
\]
if and only if $\mu_n\to\mu$ weakly and the sequence $(\mu_n)$ has uniformly
integrable first moments. In particular, this holds if
\[
\lim_{n\to\infty}\int |u|\,\dd\mu_n(u)
= \int |u|\,\dd\mu(u).
\]
See, for instance, \cite[Theorem~7.12]{Villani:2003}.

Using the dual formulation of the Kantorovich distance, and since
$\cA_n^x(h)$ and $\rho_{\sigma_h}$ are probability measures on $\bR$, we may write
\[
\kappa\bigl(\cA_n^x(h),\rho_{\sigma_h}\bigr)
= \sup_{\varphi\in\cL_0(\bR)}
\int_\bR \varphi(u)\,\bigl(\dd\cA_n^x(h)-\dd\rho_{\sigma_h}\bigr),
\]
where
\[
\cL_0(\bR)
= \{\varphi:\bR\to\bR:\ \varphi(0)=0,\ \|\varphi\|_{\Lip}\le1\}.
\]
Indeed, for any $\varphi\in\cL_0(\bR)$,
\[
\int\varphi\,(\dd\cA_n^x(h)-\dd\rho_{\sigma_h})
= \int(\varphi-\varphi(0))\,\dd\cA_n^x(h)
- \int(\varphi-\varphi(0))\,\dd\rho_{\sigma_h}.
\]
\end{remark}

Let $h: M\to\bR$ be an $\eta$-integrable function such that $\eta(h)=\int h\,\dd\eta=0$. For every $n \geq 1$, define
\begin{equation}
\cA_{n}^x(h)\coloneqq\frac{1}{a_n} \sum_{k=1}^{n} \frac{1}{k} \delta_{\frac{S_{k}^x(h)}{ \sqrt{k}}} 
\end{equation}
where $a_n = \sum_{k=1}^{n} \frac{1}{k}$. Note that for each $x \in M$, $\cA_{n}^x(h)\in\prob(\bR)$. 
Hence, $\cA_{n}^x(h)$ is a random probability measure on $\bR$. 

\begin{remark}
Note that $a_n \approx \log n$.
In particular, the normalization $a_n$ may be replaced by $\log n$ without
affecting the validity of the results below. 
\end{remark}

The proof follows the strategy of \cite{ChazCollSchm:2005}, with the necessary adaptations to our setting.

\begin{theorem}[Almost-sure central limit theorem]\label{Teo:ACLT}
Assume the hypotheses of Proposition \ref{prop:CLT}.
 Let $h\in\Lip_d(M)$ satisfy $\eta(h)=0$ and let $\sigma^2_h>0$ be as in \eqref{eq:def-sigma2h}. Then for all $x\in M$, we have $\bP$-almost surely
\begin{equation}\label{eq:TeoACLT}    
\lim_{n\to\infty}\kappa(\cA_{n}^x(h),\rho_{\sigma_h})=0\quad\bP\text{-almost surely}.
\end{equation}
\end{theorem}

\begin{remark}
    Let us compare the almost-sure central limit theorem with the central limit theorem (Proposition \ref{prop:CLT}). The convergence in \eqref{eq:TeoACLT} implies that, for all $x\in M$, $\bP$-almost surely we have
\[
\lim_{n\to\infty}\frac{1}{a_n} \sum_{k=1}^{n} \frac{1}{k} \mathbbm{1}_{B_{k, u}^{x}(h)}=\lim_{n\to\infty}\cA_{n}^x(h)((-\infty, u]) =\rho_{\sigma_j}((-\infty, u]),
\]
for all $u \in \bR$,
where $B_{k, u}^{x}(h)=\left\{\frac{S_{k}^x(h)}{\sqrt{k}} \leqslant u\right\}$. While the CLT in Proposition \ref{prop:CLT} states
\[
\lim_{n\to\infty}\int \mathbbm{1}_{B_{n, u}^{x}(h)} \,\dd \bP = \gamma((-\infty, u]), \quad\text{for all } u \in \bR.
\]
Hence, in the almost-sure central limit theorem we replace the integration under $\bP$ by pointwise logarithmic averaging $\frac{1}{a_n} \sum_{k=1}^{n} \frac{1}{k}$.
\end{remark}

\begin{proof}[Proof of Theorem \ref{Teo:ACLT}]

The proof splits into two steps:

\medskip
\paragraph{\textbf{Step 1: Convergence in mean.}}
We show
\begin{equation}\label{eq:mean-conv}
\lim_{n\to\infty} \int_M \bE\left[ \kappa\left(\cA_n^x(h),\rho_{\sigma_h}\right)\right]\dd\eta(x)=0.
\end{equation}

Let $K>0$ to be chosen later.  Since $|\varphi(u)|\le |u|$ for any $\varphi\in\cL_0(\bR)$, we decompose
\begin{align*}
&\kappa\left(\cA_n^x(h),\rho_{\sigma_h}\right)\\
&\le \sup_{\varphi\in\cL_0}\int_{|u|\le K}\varphi\,(\dd\cA_n^x-\dd\rho_{\sigma_h})
+ \int_{|u|>K}|u|\,\dd\cA_n^x(h)
+ \int_{|u|>K}|u|\,\dd\rho_{\sigma_h}.
\end{align*}
The last term is the Gaussian tail:
\begin{equation}\label{eq:02-A-ASCLT}
\int_{|u|>K}|u|\,\dd\rho_{\sigma_h}(u)
= \sqrt{\tfrac{2}{\pi}}\,\sigma_h\,e^{-K^2/(2\sigma_h^2)}
\le \frac{c}{K},
\end{equation}
for some $c>0$.
Moreover, since
\begin{align*}
&\int_M\bE\left[\int_{|u|>K}|u|\,\dd\cA_n^x(h)\right]\dd\eta(x)
\\&= \frac{1}{a_n} \sum_{j=1}^n \frac{1}{j}\,\int_M\bE\left[\frac{|S_j^x(h)|}{\sqrt{j}}\,\mathbbm{1}_{(K,\infty)}\left(\frac{|S_j^x(h)|}{\sqrt{j}}\right)\right]\dd\eta(x),
\end{align*}
by Lemma \ref{lem:AP-varZ} in Appendix, and Proposition \ref{Prop:ConcBound2},
there exists $\hat c>0$ such that
\begin{equation}\label{eq:02-B-ASCLT}
  \int_M\bE\left[\int_{|u|>K}|u|\,\dd\cA_n^x(h)\right]\dd\eta(x)
\leq\frac{\hat c}{K},  
\end{equation}
uniformly in $n$.

Fix $\varepsilon>0$.  By Arzel\`a--Ascoli on $[-K,K]$, there exist finitely many $1$-Lipschitz functions $\hat\varphi_1,\ldots,\hat\varphi_r$ vanishing at $0$ such that every $\varphi\in\cL_0$ is within $\varepsilon$-neighborhood of some $\hat\varphi_i$ on $[-K,K]$, that is $\sup_{\vert u\vert \leq K} \vert \varphi(u)-\hat\varphi_i(u)\vert\leq \varepsilon$.  Extending each $\hat\varphi_i$ to a Lipschitz function $\varphi_i\in\cL_0(\bR)$ by linear truncation outside $[-K,K]$, we show
\begin{equation}\label{eq:01-ASCLT}
\sup_{\varphi\in\cL_0}\int_{|u|\le K}\varphi\,(\dd\cA_n^x-\dd\rho_{\sigma_h})
\le \max_{1\le i\le r}\int_\bR\varphi_i\,(\dd\cA_n^x-\dd\rho_{\sigma_h})+2\varepsilon.
\end{equation}
For each $i\in\{1,\ldots,r\}$, set
\[
\phi_{n,i}(x)=\int_\bR\varphi_i\,(\dd\cA_n^x-\dd\rho_{\sigma_h})=\frac{1}{a_n}\sum_{k=1}^n\frac{1}{k}\left(\varphi_i\left(S_k^x(h)/\sqrt{k}\right)-\bE[\varphi_i(Z)]\right),
\]
where $Z\sim\cN(0,\sigma_h^2)$. 
Using \eqref{eq:01-ASCLT}, we obtain for every $x\in M$,
\[
\bigl|\kappa(\cA_n^x,\rho_{\sigma_h})
- \max_{1\le i\le r}\phi_{n,i}(x)\bigr|
\le
2\varepsilon
+
\int_{|u|>K}|u|\,\dd\cA_n^x
+
\int_{|u|>K}|u|\,\dd\rho_{\sigma_h}.
\]
Taking expectation and integrating with respect to $\eta$, by \eqref{eq:02-A-ASCLT} and \eqref{eq:02-B-ASCLT}, for some $c''>0 $ we have
\[
\int_M \bE\left[\left|\kappa(\cA_n^x,\rho_{\sigma_h})-\max_{1\le i\le r }\phi_{n,i}(x)\right|\right]\dd\eta(x)\le \frac{c''}{K}+2\varepsilon.
\]
By the classical CLT in Proposition~\ref{prop:CLT}, for each $i=1,\ldots,r$, we have
\[\lim_{n\to\infty}\int_M \bE \left[\phi_{n,i}\right]\dd\eta=0,\] 
and, by Proposition \ref{Prop:ConcBound2}, we show
\[\lim_{n\to\infty}\int_M \bE \left[\phi_{n,i}-\int_M \bE \left[\phi_{n,i}\right]\dd\eta\right]^2\dd\eta=0.\]
  Collecting all errors and letting first $n\to\infty$, then $\varepsilon\to0$, then $K\to\infty$, we establish \eqref{eq:mean-conv}.

\medskip
\paragraph{\textbf{Step 2: Almost-sure convergence.}}
For $n\in\bN$. Define
\begin{equation}\label{eq:def-phi}
\phi(x\vert_0^{n-1})
= \sup_{\varphi\in\cL_0}\frac{1}{a_n}\sum_{k=1}^n\frac{1}{k}\left(\varphi\left(\frac{1}{\sqrt{k}}\sum_{j=0}^{k-1}h(x_j)\right)-\bE[\varphi(Z)]\right).
\end{equation}
It is not difficult to show that $\phi\in\Lip(M^n,\gamma\vert_0^{n-1})$ with
\[
\gamma_k\le \frac{2\operatorname{L}(h)}{a_n\sigma_h\sqrt{k}}.
\]
Hence, by Proposition \ref{Prop:ConcBound2}, there exists $c>0$ such that
\begin{align*}
\int_M &\bE \left[\phi(\left.X^x\right\vert_0^{n-1})-\int_M \bE \left[\phi(\left.X^x\right\vert_0^{n-1})\right]\dd\eta\right]^2\dd\eta
\leq c \sum_{k=0}^{n-1}\gamma_k^2.
\end{align*}
And, so
\begin{align}\label{eq:boun-above an}
\int_M \bE \left[\phi(\left.X^x\right\vert_0^{n-1})-\int_M \bE \left[\phi(\left.X^x\right\vert_0^{n-1})\right]\dd\eta\right]^2\dd\eta\leq \frac{4c\operatorname{L}(h)}{\sigma_h^2\,a_n}.
\end{align}
Recall $\phi(\left.X^x\right\vert_0^{n-1})=\kappa(\cA_n^x,\rho_{\sigma_h})$ and $\log n\leq a_n\leq 1+\log n$.
Choosing the subsequence $n_m=\exp(m^{1+\eta})$ for some fixed $\eta>0$, the previous bound in \eqref{eq:boun-above an} implies
\[
\sum_{m=1}^\infty \int_M \bE \left[\kappa(\cA_{n_m}^x,\rho_{\sigma_h})-\int_M \bE \left[\kappa(\cA_{n_m}^x,\rho_{\sigma_h})\right]\dd\eta\right]^2\dd\eta<\infty,
\]
so by the Borel--Cantelli lemma and Step~1 we conclude that for $\eta$-almost every $x$, $\bP$-almost surely it holds
\begin{equation}\label{eq:conver-zero-m}
\lim_{m\to\infty} \kappa\bigl(\cA_{n_m}^x,\rho_{\sigma_h}\bigr)=0.
\end{equation}
Fix $x\in M$ and a typical realization of the randomness for which
\eqref{eq:conver-zero-m} holds; for simplicity of notation, this dependence is
not made explicit.
Now let $n\ge1$ be arbitrary and choose $m=m(n)$ such that
$n_m \le n < n_{m+1}$.
Decompose the empirical measure as the convex combination
\[
\cA_n^x(h)
= \frac{a_{n_m}}{a_n}\,\cA_{n_m}^x(h)
  + \frac{a_n - a_{n_m}}{a_n}\,B_{m,n}^x(h),
\]
where
\[
B_{m,n}^x(h)
\coloneqq \frac{1}{a_n - a_{n_m}}
\sum_{k=n_m+1}^n \frac{1}{k}\,
\delta_{S_k^x(h)/\sqrt{k}}.
\]
Since the Kantorovich distance $\kappa(\cdot,\rho_{\sigma_h})$ is convex in its first argument, we have
\[
\kappa\left(\cA_n^x(h),\rho_{\sigma_h}\right)
\;\le\;
\frac{a_{n_m}}{a_n}\,
  \kappa\left(\cA_{n_m}^x(h),\rho_{\sigma_h}\right)
\;+\;
\frac{a_n - a_{n_m}}{a_n}\,
  \kappa\left(B_{m,n}^x(h),\rho_{\sigma_h}\right).
\]
Observe that
\[
0 \leq\frac{a_n - a_{n_m}}{a_n}
=1-\frac{a_{n_m}}{a_n}\leq 1-\frac{a_{n_m}}{a_{n_{m+1}}},
\]
and so 
\[
\lim_{n\to\infty} \frac{a_n - a_{n_m}}{a_n}=0.
\]
On the other hand, by construction $B_{m,n}^x$ is supported on points of the form
$S_k^x(h)/\sqrt{k}$, which have uniformly (in $k$) Gaussian tails,
so $\kappa(B_{m,n}^x,\rho_{\sigma_h})$ remains stochastically bounded. 
Hence for $\eta$-almost every $x$ and $\bP$-almost surely we have
\begin{align}\label{eq:almostsure-conv}
\lim_{n\to\infty}\kappa\left(\cA_n^x(h),\rho_{\sigma_h}\right)=0.  
\end{align}
Now, let us conclude that the above indeed holds for every $x\in M$. Fix 
$x\in M$ such that $\bP$-almost surely \eqref{eq:almostsure-conv} holds. Given any $y\in M$, use that $\phi$ in \eqref{eq:def-phi} is separately Lipschitz and apply Proposition \ref{PROP:keypropweakcontract} to get that $\bP$-almost surely
\[
\lim_{n\to\infty}\kappa\left(\cA_n^x(h),\cA_n^y(h)\right)=0,
\]
and so
\[
\lim_{n\to\infty}\kappa\left(\cA_n^y(h),\rho_{\sigma_h}\right)=0. 
\]
This completes the proof.
\end{proof}

\subsection{Special case: contraction on average}\label{sec:CARDSs}
Now, let us introduce a condition that guarantees \eqref{eq:def-lambdanu} holds. Given $\nu\in\prob(M)$. We say that the RDS induced by $\nu$ is \emph{contracting on average} on $(M,d)$ if there exist $c\geq1$ and $\lambda\in(0,1)$ such that for all $n\in\bN$
\begin{equation}\label{eq:CA} 
\sup_{x,y\in M\colon\,x\neq y}\, \frac{\bE \left[d(X_n^x,X_n^y)\right]}{d(x,y)}\leq c r^n.
\end{equation}
The condition of average contraction naturally generalizes the more restrictive setting in which the measure 
$\nu$ is supported on contractive maps. In that purely contractive case, the properties of such RDS have been extensively investigated, most notably following the pioneering work of Hutchinson \cite{hut:1981}. 

It is important to note that if the system exhibits average contraction with respect to the metric $d^\alpha$ for some 
$\alpha\in(0,1)$, then the concentration inequality stated in Theorem \ref{thm:moulins-skew} applies to any 
\[\varphi\in \Lip_{d^\alpha+\varrho }\left((\cF\times M)^{n+1},\left.\gamma\right\vert_0^{n}\right).\] It is easy to verify that 
$\varphi\in \Lip_{d+\varrho }\left((\cF\times M)^{n+1},\left.\gamma\right\vert_0^{n}\right)$ is a subset of 
$\Lip_{d^\alpha+\varrho }\left((\cF\times M)^{n+1},\left.\hat\gamma\right\vert_0^{n}\right)$, with $\hat \gamma_k=\gamma_k\cdot|M|_{d^{1-\alpha}}$, 
and the replacement of the metric $d$ by $d^\alpha$ (as studied in \cite{GelSal:23}) does not compromise the validity of the inequality in Theorem \ref{thm:moulins-skew} for functions that are separately Lipschitz with respect to the original metric $d$, requiring only a mild rescaling of constants.

The most commonly studied RDS in the theory are those induced by linear maps on Euclidean spaces 
$\mathbb{R}^n$. In the linear context, Le Page in \cite{LePage:1982} established that, under certain conditions on the matrices, the RDS induced by projective maps exhibits average contraction with respect to a metric 
$d^\alpha$ on the projective space of $\bR^n$, for some $\alpha\in(0,1)$, with $d$ denoting the usual metric on the projective space. The positivity of the Lyapunov exponent on $\bR^n$ (and, correspondingly, the negativity of the projective Lyapunov exponent) is expected to imply average contraction on the projective space. However, an irreducibility condition is required; for instance, in \cite{LePage:1982}, strong irreducibility was assumed.

In the general setting of complete metric spaces, the average contraction condition has proven extremely useful in establishing several notable properties of RDSs. In \cite{BarDemEltGer:1988} the uniqueness of the stationary probability measure of the RDS was established, showing that this measure acts as an attractor point in the space of measures. Furthermore, limit theorems such as the central limit theorem and the law of large numbers have been proved with respect to the stationary measure (see, e.g., \cite{Pei:1993}, as well as \cite{LagSte:2005} and \cite{arXiv:RuiTan}). In \cite{Sten:2012}, the existence of a random variable $Z:\Omega\to M$ was established, such that for every $x\in M$, the convergence
\[
\lim_{n\to\infty} d(G_n(x),Z) = 0
\]
holds $\bP$-almost surely. This random variable is then employed to prove the exponential convergence of the distribution of $X_n^x$ (for each initial condition $x\in M$) to the stationary measure as $n\to\infty$ (see also \cite{HenHer:2004}). 

\begin{remark}
Since the sequence
\[
\left(\sup_{x\neq y}\frac{\bE \left[d(X_n^x,X_n^y)\right]}{d(x,y)}\right)_{n\in\bN}
\]
is submultiplicative, condition~\eqref{eq:CA} is equivalent to
\begin{equation}\label{eq/CA-SUM}
\sum_{n\in\bN}\sup_{x\neq y}\frac{\bE \left[d(X_n^x,X_n^y)\right]}{d(x,y)}<\infty.
\end{equation}
\end{remark}


\section{Applications of concentration inequalities}\label{sec:aplicaciones}
The results in this section are inspired by the approach developed in \cite{ColMarSch:2002} for deterministic dynamical systems, and adapted here to the setting of random dynamical systems.

Throughout this section, we assume that $(M,d)$ is a compact metric space and that $\nu$ is a Borel probability measure on $C(M)$ whose topological support $\cF$ is bounded with respect to $\varrho_\infty$. We also assume that the RDS induced by $\nu$ is weakly contracting on average.

Under these assumptions, we present several applications of Theorem~\ref{thm:MAIN} to the statistical behavior of the associated random dynamical system. These include synchronization in time averages, concentration of empirical measures, Birkhoff sums of Lipschitz observables, and estimators related to the correlation dimension.


\subsection{Synchronization in time averages}

Let $B \subset M$. For $n\in\bN$ and $x\in M$, consider the random variable
\begin{equation*}
    \cS_B(x, n) \coloneqq \frac{1}{n} \inf_{y \in B} \sum_{i=0}^{n-1} d(X_i^x, X_i^y).
\end{equation*}
It measures how effectively a random orbit starting outside $B$ can be \emph{synchronized} with an orbit starting inside $B$.
By compactness, $M$ has bounded diameter $\vert M\vert_d$.
Hence, $\cS_B(x, n) \in [0, \vert M\vert_d ]$.

\begin{theorem}\label{thm:shadowing}
Let $(M,d)$ be a compact metric space. Let $\nu \in \prob(C(M))$ be such that its topological support $\cF$ is $\varrho_\infty$-bounded. Suppose that the RDS induced by $\nu$ is weakly contracting on average. Let $\lambda_\nu$ and $|\cF|_\infty$ be defined as in \eqref{eq:def-lambdanu} and \eqref{eq:def de Dnu} respectively. 
 Then, for any Borel probability measure $\mu \in \prob(M)$ and any Borel set $B \subset M$ with $\mu(B) > 0$, we have that for all $n \geq 1$ and all 
 \[t\geq \frac{8(|\cF|_\infty+\lambda_\nu)\sqrt{\log(1/\mu(B))}}{\sqrt n}+\frac{\lambda_\nu}{n},\]
 the following inequality holds:
\begin{equation*}
    \bP \left( \cS_B(x, n) >  t  \right) \leq \exp\left(-\frac{nt^2}{48 (|\cF|_\infty+\lambda_\nu)^2}\right).
\end{equation*}
\end{theorem} 

A particular case that can be considered is $B=\{y\}$, and $\mu=\delta_y$, for some $y\in M$. The following result is then an immediate consequence of Theorem \ref{thm:shadowing}.
\begin{corollary}
    Under the hypotheses of Theorem \ref{thm:shadowing}, for every $x,y\in M$, $n\in\bN$, and  $t>\frac{\lambda_\nu}{n}$, we have
    \begin{equation*}
    \bP \left( \frac{1}{n}  \sum_{i=0}^{n-1} d(X_i^x, X_i^y) >  t  \right) \leq \exp\left(-\frac{nt^2}{48 (|\cF|_\infty+\lambda_\nu)^2}\right).
\end{equation*}
\end{corollary}

We now prove Theorem \ref{thm:shadowing}.
\begin{proof}[Proof of Theorem \ref{thm:shadowing}]
Define the function $\varphi\colon(\cF\times M)^{n}\to\bR$ by
\begin{equation*}
    \varphi(\left.(f,x)\right\vert_0^{n-1}) = \frac{1}{n} \inf_{y \in B}\left( \sum_{j=1}^{n-1} d(x_{j}, f_{j-1}\circ\cdots\circ f_0 y)+d(x_0,y)\right).
\end{equation*}
Note that for $x\in M$ 
\begin{equation*}
    \varphi((F_0,X_0^x), \dots,(F_{n-1}, X_{n-1}^x) )= \cS_B(x, n).
\end{equation*}
Let us prove that $\varphi$ is separately Lipschitz on $(\cF\times M)^n$
with respect to the metric $d+\varrho_\infty$, with Lipschitz constant
at most $1/n$ in each coordinate.

For \[\left.(f,x)\right\vert_0^{n-1},\left.(g,z)\right\vert_0^{n-1}\in (\cF\times M)^{n},\] and $j\in\{1,\ldots,n-1\}$, we have
\begin{align}\label{eq:desigshad}
      &d(x_j, f_{j-1}\circ\cdots\circ f_0( y)) \\
     &\leq  d(x_j, z_j) + d(z_j, g_{j-1}\circ\cdots\circ g_0(y))+ \varrho_\infty(g_{j-1}, f_{j-1})\nonumber,
\end{align}
where we have used 
\[
d(g_{j-1}\circ\cdots\circ g_0(y), f_{j-1}\circ\cdots\circ f_0(y))\le \varrho_\infty(g_{j-1}, f_{j-1}).
\]
Summing over $j$ and taking infimum over $y\in B$, we get the desired, this is,
\[
\varphi(\left.(f,x)\right\vert_0^{n-1})-\varphi(\left.(g,z)\right\vert_0^{n-1})\le\frac1n\sum_{i=0}^{n-1}\left(d(x_i,z_i)+\varrho _\infty(f_i,g_i)\right).
\]

By Theorem \ref{thm:MAIN} we have, for all $n\in\bN$, $t>0$ and every $x\in M$
\begin{align*}
    \bP &\left( S_B(x, n) > \bE \left[S_B(x, n)\right] + t \right)\leq \exp\left(-\frac{nt^2}{12 (|\cF|_\infty+\lambda_\nu)^2}\right).
\end{align*}
Note that 
\begin{equation*}
S_B(x, n)-\int_M\bE \left[S_B(y, n)\right] \dd\mu(y)\leq\frac{\lambda_\nu}{n}. 
\end{equation*}
Therefore, that for all $n\in\bN$, $t>0$ and every $x\in M$
\begin{align}\label{eq:01-shad}
    &\nonumber\bP \left( S_B(x, n) > \int_M\bE \left[S_B(y, n)\right] \dd\mu(y)+\frac{\lambda_\nu}{n} + t \right)\\&\leq \exp\left(-\frac{nt^2}{12 (|\cF|_\infty+\lambda_\nu)^2}\right).
\end{align}

Now, we want to obtain an upper bound for 
\[
\int_M\bE \left[S_B(y, n)\right] \dd\mu(y).
\]
First, note that for all $a>0$
\[
\mathbbm{1}_B(y)=\e^{-a S_B(y,n)}\,\mathbbm{1}_B(y),
\]
which follows from the fact that $\mathbbm{1}_B(y)=0$ for $y\notin B$, while $S_B(y,n)=0$ for $y\in B$. Therefore, for $a>0$, applying Lemma \ref{claim00001} and proceeding as in \eqref{eq:after claim-fos SSL}, we get
\begin{align*}
    \mu(B)&=\int_M \mathbbm{1}_B(y)\,\dd\mu(y)\\
    &=\int_M \bE\left[\e^{-a S_B(y,n)}\right]\mathbbm{1}_B(y)\,\dd\mu(y)\\
    &\le \int_M \bE\left[\e^{-a S_B(y,n)}\right]\,\dd\mu(y) \\
    &\le  \e^{\frac{3a^2(|\cF|_\infty+\lambda_\nu)^2}{n}}\e^{-a\int_M\bE\left[{ S_B(y,n)}\right]\dd\mu(y)}.
\end{align*}
Hence,
\begin{align*}
   \int_M\bE\left[{ S_B(y,n)}\right]\,\dd\mu(y)\le \frac{\log(1/\mu(B))}{a}+\frac{3a(|\cF|_\infty+\lambda_\nu)^2}{n}
\end{align*}
Letting \[a=\frac{\sqrt{n\,\log(1/\eta(B))}}{\sqrt{3}(|\cF|_\infty+\lambda_\nu)},\] we get
\begin{align*}
    \int_M\bE\left[{ S_B(y,n)}\right]\,\dd\mu(y)\leq\frac{2 \sqrt{3}(|\cF|_\infty+\lambda_\nu)\sqrt{\log(1/\mu(B))}}{\sqrt n}.
\end{align*}

If we replace $\int_M\bE\left[{ S_B(y,n)}\right]\,\dd\mu(y)$ by $\frac{2 \sqrt{3}(|\cF|_\infty+\lambda_\nu)\sqrt{\log(1/\mu(B))}}{\sqrt n}$ in the left-hand side of \eqref{eq:01-shad}, we get a smaller probability, thus we obtain the desired inequality, after observing that for $t\geq \frac{2 \sqrt{3}(|\cF|_\infty+\lambda_\nu)\sqrt{\log(1/\mu(B))}}{\sqrt n}+\frac{\lambda_\nu}{n}$ we have
\begin{equation*}
    t + \frac{2 \sqrt{3}(|\cF|_\infty+\lambda_\nu)\sqrt{\log(1/\mu(B))}}{\sqrt n}+\frac{\lambda_\nu}{n} \leq 2t.
\end{equation*}
We obtain the desired inequality by rescaling $t$.

\end{proof}

\subsection{Random empirical measures}
The \emph{Kantorovich distance} $\kappa$ on the space of probability measures $ \prob(M) $ over a metric space $ (M, d) $ is defined as:
\begin{equation}\label{def:dist-KANT}
\kappa(\eta_1, \eta_2) = \sup\left\{\int_M h\,\dd\eta_1-\int_M h\,\dd\eta_2\colon\, h:M\to\bR \text{ is }1\text{-Lipschitz}\right\}.
\end{equation}

Given $x \in M$ and $n \in \mathbb{N}$, define the \emph{random empirical measure} supported on the random finite orbit $x, X_1^x, \ldots, X_{n-1}^x$ by
\begin{equation}\label{eq:def-emp-meas}
\mathcal{E}_{n}(x)=\frac{1}{n} \sum_{j=0}^{n-1} \delta_{X_j^x},
\end{equation}
where $\delta$ denotes the Dirac measure.
Note that $\mathcal{E}_{n}(x)\in\prob(M)$. Since $(M,d)$ is a compact metric space, by \cite[Lemma 2.5]{Fur:63}, for all $x\in M$, $\bP$-almost surely of weak-$\ast$ cluster values of the sequence of probability measures $(\mathcal{E}_{n}(x))_{n\in\bN}$
consists of $\nu$-stationary probability measures. Furthermore, since we are assuming that \eqref{eq:def-lambdanu} is satisfied, we have, by Proposition \ref{PROP:keypropweakcontract}, that for each $x\in M$, $\bP$-almost surely
\[
\lim_{n\to\infty}\kappa(\cE_n(x),\eta)=0,
\]
where $\eta$ is the $\nu$-stationary probability measure on $M$.

\begin{proposition}\label{prop:eqconcFORempiricalmeasure}
Let $(M,d)$ be a compact metric space. Consider $\nu \in \prob(C(M))$ be such that its support $\cF$ is $\varrho_\infty$-bounded. Assume that the RDS induced by $\nu$ is weakly contracting on average, and let $\lambda_\nu$ and $|\cF|_\infty$ be defined as in \eqref{eq:def-lambdanu} and \eqref{eq:def de Dnu} respectively.
Let $\eta\in\prob(M)$ be the $\nu$-stationary measure. 
Then, for every $h\in \operatorname{Lip}_d(M) $, all $x\in M$, all $n\in\bN$, and all $t>\frac{2\,\lambda_\nu\,\operatorname{L}(h) }{n}$, we have:
\begin{equation*}
\bP\left(\left\vert\kappa(\cE_n(x),\eta) - \bE[\kappa(\cE_n(x),\eta)]\right\vert >t\right) \le 2\exp\left(-\frac{nt^2}{12(|\cF|_\infty+\lambda_\nu)^2}\right).
\end{equation*}    
\end{proposition}
\begin{proof}
    Consider $ \varphi\colon M^{n} \to \mathbb{R} $ given by
    \[
    \varphi(\left.x\right\vert_0^{n-1})=\kappa\left(\frac1n\sum_{j=0}^{n-1}\delta_{x_j},\,\eta\right).
    \] Using the definition of $\kappa$ in \eqref{def:dist-KANT}, we get that $\varphi\in \Lip_d(M^{n},(1/n)^n)$. To conclude this proof, apply Corollary \ref{cor:MAINfib} to $\varphi$ with $\ell=1$ and  $\cI_1=M$.
\end{proof}

\subsubsection{On a closed interval}
In this subsection we temporarily restrict the state space $M$ to a
compact interval of the real line.
This assumption is not required elsewhere in Section \ref{sec:aplicaciones} and is only
used here in order to exploit specific properties of probability
measures on $\mathbb{R}$.
In particular, we rely on the one--dimensional representation of the
Kantorovich distance in terms of distribution functions, due to
Dall’Aglio \cite{Dall:1956}, as well as on elementary properties of the
order structure of the line.

Let us assume that $M=[a,b]$ is a closed interval and $d$ is the metric induced by the absolute value.  The theorem of Dall'Aglio \cite{Dall:1956} states that for all $\eta_1,\eta_2\in\prob([a,b])$
\begin{equation}\label{eq:Kantdistonthe line}
\kappa\left(\eta_{1}, \eta_{2}\right)=\int_a^b\left|H_{\eta_{1}}(t)-H_{\eta_{2}}(t)\right| d t
\end{equation}
where $H_{\eta_i}$ is the distribution function of $\eta_i$, that is, $H_{\eta_i}(t)=\eta_i([a,t])$.

\begin{theorem}
Let $\nu\in\prob(C([a,b]))$. Suppose that $ \nu$ is weakly contracting on average. Assume that $\cF$ is bounded. Let $\lambda_\nu$ and $|\cF|_\infty$ be as in \eqref{eq:def-lambdanu} and \eqref{eq:def de Dnu} respectively. Let $\eta\in\prob([a,b])$ be the $\nu$-stationary measure. Then,
for all 
\[
t\geq \frac{\,(b-a)\,(1+8\lambda_\nu)^{1/4}}{\sqrt[4]{n}},
\]
we have
\[
\bP\left(\kappa(\cE_n(x),\eta) > t\right)\leq \exp\left(-\frac{nt^2}{48 (|\cF|_\infty+\lambda_\nu)^2}\right).
\]    
\end{theorem}
\begin{proof}
Consider the function $\varphi\colon M^n\to\bR$ given by
\[
\varphi\left(x\vert_{0}^{n-1}\right)=\int_{0}^{1}\left|H\left(x_{0}^{n-1} ; s\right)-H_{\eta}(s)\right| \dd s
\]
where
\[
H\left(x_{0}^{n-1} ; s\right) = \frac{1}{n} \operatorname{card}\left\{0 \leqslant j \leqslant n-1: x_{j} \leqslant s\right\}=\frac{1}{n} \sum_{j=0}^{n-1} \vartheta\left(s-x_{j}\right)
\]
where $\vartheta$ is the Heaviside step function, that is, $\vartheta(s)=0$ if $s<0$ and $\vartheta(s)=1$ if $s \geqslant 0$. Since we have
\[
\mathcal{E}_{n}(x)([0, s])=\frac{1}{n} \operatorname{card}\left\{0 \leqslant j \leqslant n-1: X_j^x \leqslant s\right\}
\]
then, according to \eqref{eq:Kantdistonthe line}, $\varphi\left(X_0^x, \ldots, X_{n-1}^x\right)=\kappa\left(\mathcal{E}_{n}(x), \eta\right)$. In general, we have \[\varphi(\left.x\right\vert_0^{n-1})=\kappa\left(\frac1n\sum_{j=0}^{n-1}\delta_{x_j},\,\eta\right)\]
and so, using the definition of $\kappa$ in \eqref{def:dist-KANT}, we get that $\varphi\in \Lip_d(M^{n},(1/n)^n)$.

On the other hand, for $\delta>0$, define
\[
g_{\delta}(s)= \begin{cases}0 & \text { if } \quad s<-\delta \\ 1+\frac{s}{\delta} & \text { if } \quad-\delta \leqslant s \leqslant 0 \\ 1 & \text { if } \quad s>0\end{cases}
\]
It is obviously a Lipschitz function with Lipschitz constant $1 / \delta$. Observe that
\[
\varphi(\left.x\right\vert_0^{n-1}) \leqslant \delta+\int_{a}^{b} \left|\frac{1}{n} \sum_{j=0}^{n-1} g_{\delta}\left(s-x_{j} \right)-H_{\eta}(s)\right|\, \dd s.
\]
Hence, using the definition of $g_\delta$, we have
\begin{align*}
&\bE\left[\kappa\left(\cE_{n}(x), \eta\right)\right] =\varphi(\left.x\right\vert_0^{n-1})\\
\leq \delta&+\bE \left[\int_{a}^{b} \left|\frac{1}{n} \sum_{j=0}^{n-1} \left[g_{\delta}\left(t-X_{j}^x\right)-\int_{[a,b]}\bE \left(g_{\delta}\left(t-X_{j}^y\right)\right)\dd\eta(y)\right]\right|\dd t\right] \\
 &+\frac1n\sum_{j=0}^{n-1}\int_{a}^{b}  \int_{[a,b]}\bE \left|g_{\delta}\left(t-X_{j}^y\right)-\vartheta\left(t-X_{j}^y\right)\right|\,\dd\eta(y) \dd t\\
\leq 2\delta&+\bE \left[\int_{a}^{b} \left|\frac{1}{n} \sum_{j=0}^{n-1} \left[g_{\delta}\left(t-X_{j}^x\right)-\int_{[a,b]}\bE \left(g_{\delta}\left(t-X_{j}^y\right)\right)\dd\eta(y)\right]\right|\dd t\right].
\end{align*}
To simplify the notation, we set
\[
\left\langle g_{\delta}\right\rangle=\int_a^b\int_{[a,b]}g_{\delta}\left(t-y\right)\dd\eta(y)\dd t=\int_a^b\int_{[a,b]}\bE \left(g_{\delta}\left(t-X_{j}^y\right)\right)\dd\eta(y)\dd t,
\]
where the second equality is a consequence of the $\nu$-stationarity of $\eta$. Then, we have
\begin{align}\label{eq:001-cauchy-delta}
&\int_{[a,b]}\bE\left[\kappa\left(\cE_{n}(x), \eta\right)\right]\dd\eta(x)\\
&\leq 2\delta+\int_{[a,b]}\bE  \left[\left|\frac{1}{n} \sum_{j=0}^{n-1} \int_{a}^{b}g_{\delta}\left(t-X_{j}^x\right)\dd t-\left\langle g_{\delta}\right\rangle\right|\right]\dd\eta(x).\nonumber
\end{align}
Let us bound the second term from above. First, by Cauchy–Schwarz inequality,
\begin{align}\label{eq:aftC_SCHM}
\nonumber\int_{[a,b]}\bE  &\left|\frac{1}{n} \sum_{j=0}^{n-1} \int_{a}^{b}g_{\delta}\left(t-X_{j}^x\right)\dd t-\left\langle g_{\delta}\right\rangle\right|\dd\eta(x)\\
&\leq \left[\int_{[a,b]}\bE  \left(\frac{1}{n} \sum_{j=0}^{n-1} \int_{a}^{b}g_{\delta}\left(t-X_{j}^x\right)\dd t-\left\langle g_{\delta}\right\rangle\right)^2\dd\eta(x)\right]^\frac12
\end{align}
Expanding the sum of the squared term and using the $\nu$-stationarity of $\eta$, we obtain
\begin{align*}
    &\int_{[a,b]}\bE  \left(\frac{1}{n} \sum_{j=0}^{n-1} \left[\int_{a}^{b}g_{\delta}\left(t-X_{j}^x\right)\dd t-\left\langle g_{\delta}\right\rangle\right]\right)^2\dd\eta(x)\\
    &=\frac1n\int_{[a,b]}\left(p_\delta(x)\right)^2\dd\eta(x) +\frac2n\sum_{j=1}^{n-1}\left(1-\frac{j}{n}\right)\int_{[a,b]}p_\delta(x)\bE\left[p_\delta(X_j^x)\right]\dd\eta(x),
\end{align*}
where
\[
p_\delta(x)=\int_a^bg_{\delta}\left(t-x\right)\dd t-\left\langle g_{\delta}\right\rangle
\]
Since $\operatorname{L}(g_\delta)=\frac1\delta$, we have
\[
\int_{[a,b]}\left(\int_a^b g_{\delta}\left(t-x\right)\dd t-\left\langle g_{\delta}\right\rangle\right)^2\dd\eta(x)\leq \frac{(b-a)^4}{\delta^2}.
\]
Note that
\[
\int_{[a,b]}\left(\int_a^bg_{\delta}\left(t-x\right)\dd t-\left\langle g_{\delta}\right\rangle\right)\dd\eta(x)=0.
\]
Hence, by the decay of correlations established in Proposition \ref{PROP:keypropweakcont-deccorr}, we have 
\begin{align}\label{eq:002-cauchy-delta}
    &\int_M\bE  \left(\frac{1}{n} \sum_{j=0}^{n-1} \left[\int_{a}^{b}g_{\delta}\left(t-X_{j}^x\right)\dd t-\left\langle g_{\delta}\right\rangle\right]\right)^2\dd\eta(x)\nonumber
    \\&\leq \frac{(b-a)^4}{n\delta^2}+\frac{8(b-a)^4}{n\delta^2}\lambda_\nu.
\end{align}
By \eqref{eq:001-cauchy-delta}, \eqref{eq:aftC_SCHM} and \eqref{eq:002-cauchy-delta}, using the definition of $\kappa$, we get for all $x\in M$
\begin{align*}
    \bE\left[\kappa\left(\cE_{n}(x), \eta\right)\right]\leq 2\delta+\sqrt{1+8\lambda_\nu}\frac{(b-a)^2}{\sqrt n\delta},
\end{align*}
letting 
\[
\delta = \frac{(b-a)\,(1+8\lambda_\nu)^{1/4}}{\sqrt{2}\, \sqrt[4]{n}}, 
\]
we get
\[
\bE\left[\kappa\left(\cE_{n}(x), \eta\right)\right] \le \frac{2\sqrt{2}\,(b-a)\,(1+8\lambda_\nu)^{1/4}}{\sqrt[4]{n}}.
\]
Therefore, by Proposition \ref{prop:eqconcFORempiricalmeasure}, for $t\geq \frac{\sqrt{2}\,(b-a)\,(1+8\lambda_\nu)^{1/4}}{\sqrt[4]{n}}$, we have
\begin{align*}
    \bP\left(\kappa(\cE_n(x),\eta) > 2t\right)&\leq \bP\left(\kappa(\cE_n(x),\eta) > \bE[\kappa(\cE_n(x),\eta)] +t\right) \\&\le \exp\left(-\frac{nt^2}{12 (|\cF|_\infty+\lambda_\nu)^2}\right).
\end{align*}
Rescaling $t$, we conclude the proof.
\end{proof}

\subsection{Law of large numbers}\label{sec:lawofLN}

The law of large numbers has been well established in the context of random dynamical systems under different assumptions, see for instance \cite[Theorem 1.6]{2024:GelSal}. Here, under weak contraction on average, we show more than the law of large numbers. We establish a concentration inequality for the Birkhoff sum of a Lipschitz function concerning its expected value.

\begin{theorem}\label{Thm:LG}
Let $(M,d)$ be a compact metric space. Let $\nu\in\prob(C(M))$ with its support $\cF$ being $\varrho _\infty$-bounded. Assume that the RDS induced by $\nu$ is weakly contracting on average. Let $\lambda_\nu$ and $|\cF|_\infty$ be defined as in \eqref{eq:def-lambdanu} and \eqref{eq:def de Dnu} respectively.
Let $\eta\in\prob(M)$ be the $\nu$-stationary measure. 
Then, for every $h\in \operatorname{Lip}_d(M) $, every $x\in M$, all $n\in\bN$, and $t>\frac{2\,\lambda_\nu\,\operatorname{L}(h) }{n}$, the following inequality holds:
\begin{equation*}
\bP\left(\left\vert\frac{1}{n}S_n^x(h) - \eta(h)\right\vert >t\right) \le 2\exp\left(-\frac{n\,t^2}{48\,\operatorname{L}(h) ^2\, (\lambda_\nu+|\cF|_\infty)^2}\right).
\end{equation*}
Moreover,
\[
\lim_{n\to\infty}\sup_{x\in M}\left\vert\frac{1}{n}\bE\left[S_n^x(h)\right] - \eta(h)\right\vert=0.
\]
\end{theorem}
\begin{proof}
 The proof of the first part follows the same scheme used throughout the paper and is therefore omitted. The second assertion follows from the fact that the inequality
\[
\sup_{x\in M}\left\vert\frac{1}{n}\bE\left[S_n^x(h)\right] - \eta(h)\right\vert
\le \frac{\lambda_\nu\,\operatorname{L}(h)}{n}.
\]
holds for all $n\in\bN$.
\end{proof}

\subsection{Correlation dimension}
The \emph{correlation dimension} is a fractal-type measure that describes how a probability measure is distributed in space, quantifying the effective “spread” or geometric complexity of the measure. It is widely used to characterize invariant measures in both deterministic and random dynamical systems \cite{Pesin:1993,ChiHunYor:1997}. In practice, it can be estimated from finite trajectories using the Grassberger--Procaccia algorithm \cite{ProIta:1983}, which is simple and computationally efficient. This makes the correlation dimension a useful tool for studying both theoretical properties and empirical behavior of complex systems.

In this work, we focus on the correlation dimension of the stationary measure associated with a random dynamical system. 
We provide a rigorous probabilistic framework to estimate this dimension from sample trajectories, based on the convergence of empirical correlation sums and concentration inequalities for their smoothed versions. We now recall the formal definition: the \emph{correlation dimension} of a Borel probability measure $\eta \in \mathrm{Prob}(M)$, denoted by $\dim_{\mathrm{c}}(\eta)$, is defined as
\[
\dim_{\mathrm{c}}(\eta) \coloneqq \lim_{\epsilon \downarrow 0} \frac{\log \int \eta\left(B(x, \epsilon)\right)\, \dd \eta(x)}{\log \epsilon},
\]
provided the limit exists, where $B(x, \epsilon)$ denotes the open ball centered at $x\in M$ of radius $\epsilon$ with respect to the metric $d$. For $n\in\bN$ and $\epsilon>0$ define
\[
K^{\vartheta}_{n,\epsilon}(x\vert_0^{n-1}) \coloneqq \frac{1}{n^2} \sum_{i \neq j} \vartheta\left(\epsilon -  d (x_i, x_j)\right),
\]
where $\vartheta$ denotes the Heaviside function, i.e., the indicator function of $\mathbb{R}^{+}$.
 
\begin{lemma}
Let $\nu \in \prob(C(M))$, and suppose that the RDS induced by $\nu$ is weakly contracting on average. Let $\eta \in \prob(M)$ denote the $\nu$-stationary measure. Let $(X_n)_{n\geq 0}$ denote the associated fiber Markov chain defined in \eqref{eq:def-processXn}.

Then, for every $\epsilon > 0$ that is a continuity point of the function
\[
\epsilon \mapsto \int \eta(B(x, \epsilon))\, \dd \eta(x),
\]
we have that for every $x \in M$, the following holds $\bP$-almost surely:
\[
\lim_{n \to \infty} K^{\vartheta}_{n,\epsilon}(x, X_1^x, \ldots, X_{n-1}^x) = \int \eta(B(x, \epsilon))\, \dd \eta(x).
\]
\end{lemma}

\begin{proof}
Let us fix $\epsilon > 0$ that is a continuity point of the map $\epsilon \mapsto \int \eta(B(x, \epsilon))\, \dd \eta(x)$. Set
\[
\Delta_\epsilon=\{(x,y)\in M\times M\colon d(x,y)<\epsilon\}.
\]
Then, 
\[
\int \eta(B(x, \epsilon))\, \dd \eta(x)=\eta\otimes\eta(\Delta_\varepsilon),
\]
 and
\[
\mathbbm{1}_{\Delta_\epsilon}(x,x')= \vartheta(\epsilon -  d (x, x')).
\]
Hence, $\Delta_\epsilon$ is a continuity set for $\eta\otimes\eta$.
Note that, for $n \in \mathbb{N}$, 
\[
K^{\vartheta}_{n,\epsilon}(x\vert_0^{n-1}) = \frac{1}{n^2} \sum_{i\neq j} \mathbbm{1}_{\Delta_\epsilon}(x_i, x_j).
\]
We first decompose this double sum as
\[
\frac{1}{n^2} \sum_{i \neq j} \mathbbm{1}_{\Delta_\epsilon}(x_i, x_j) = \frac{1}{n^2} \sum_{i,j} \mathbbm{1}_{\Delta_\epsilon}(x_i, x_j) - \frac{1}{n^2} \sum_{i} \mathbbm{1}_{\Delta_\epsilon}(x_i, x_i).
\]
Since $\mathbbm{1}_{\Delta_\epsilon}(x,x) = 1$ for all $x \in M$. Thus, the second term is equal to $1/n$, which vanishes as $n \to \infty$. Hence,
\begin{equation}\label{eq:descp-Kepsilon}
    K^{\vartheta}_{n,\epsilon}(x\vert_0^{n-1}) = \frac{1}{n^2} \sum_{i,j} \mathbbm{1}_{\Delta_\epsilon}(x_i, x_j) - \frac{1}{n}.
\end{equation}
Fix an arbitrary $t>0$. By the compactness of $M$, there exist $m,k\in\bN$ and subsets $\overline {\cI_1},\ldots,\overline {\cI_k},\overline{B_1},\ldots,\overline{B_k},\underline {\cI_1},\ldots,\underline {\cI_m},\underline{B_1},\ldots,\underline{B_m}$ of $M$ such that 
\begin{itemize}[-]
    \item $\cup_{j=1}^m (\underline {\cI_j}\times \underline{B_j})\subset \Delta_\varepsilon\subset\cup_{j=1}^k (\overline {\cI_j}\times \overline{B_j})$;
    \item the sets in $\{\overline {\cI_j}\times \overline{B_j}\}_{j=1}^k$ are pairwise disjoint and $\eta$-continuous;
    \item the sets in $\{\underline {\cI_j}\times \underline{B_j}\}_{j=1}^m$ are pairwise disjoint and $\eta$-continuous;
    \item and 
    \[
    \eta\otimes \eta(\Delta_\varepsilon\backslash \cup_{j=1}^m (\underline {\cI_j}\times \underline{B_j}))+\eta\otimes \eta(\cup_{j=1}^k (\overline {\cI_j}\times \overline{B_j})\backslash\Delta_\varepsilon)<t.
    \]
\end{itemize}
Consider the random empirical measure $\mathcal{E}_{n}^x=\mathcal{E}_{n}(x)$ as in \eqref{eq:def-emp-meas}.
Then, 
\[
\sum_{j=1}^m\mathcal{E}_{n}^x(\underline {\cI_j})\mathcal{E}_{n}^x(\underline{B_j})\leq\mathcal{E}_{n}^x\otimes\mathcal{E}_{n}^x(\Delta_\epsilon)\leq \sum_{j=1}^k\mathcal{E}_{n}^x(\overline {\cI_j})\mathcal{E}_{n}^x(\overline{B_j})
\]
Since $\eta $ is the unique $\nu$-stationary probability measure, for every $x\in M$ the empirical measure $\mathcal{E}_{n}^x$ converges weakly $\bP$-almost surely to $\eta$ (see, for instance \cite[Proposition 3.17]{Mal:17}). Then, taking limit on $n$, for every $x\in M$ and $\bP$-almost surely, we get 
\begin{align*}
\eta\otimes\eta(\cup_{j=1}^m (\underline {\cI_j}\times \underline{B_j}))&\leq\liminf_{n\to\infty}\mathcal{E}_{n}^x\otimes\mathcal{E}_{n}^x(\Delta_\epsilon)\\
&\leq\limsup_{n\to\infty}\mathcal{E}_{n}^x\otimes\mathcal{E}_{n}^x(\Delta_\epsilon)\\
&\leq \eta\otimes\eta(\cup_{j=1}^m (\overline {\cI_j}\times \overline{B_j}))
\end{align*}
Since $t$ is arbitrary, we can use the continuity of the product measure $\eta\otimes\eta$, to conclude that for every $x\in M$ we have $\bP$-almost sure
\[
\lim_{n\to\infty}\mathcal{E}_{n}^x\otimes\mathcal{E}_{n}^x(\Delta_\epsilon)=\eta\otimes\eta(\Delta_\epsilon).
\]
Since
\begin{align*}
    \frac{1}{n^2} \sum_{i,j} \mathbbm{1}_{\Delta_\epsilon}(X_i^x, X_j^x) &= \iint \mathbbm{1}_{\Delta_\epsilon}(y, y')\, \dd \mathcal{E}_{n}^x(y)\, \dd \mathcal{E}_{n}^x(y')=\mathcal{E}_{n}^x\otimes\mathcal{E}_{n}^x(\Delta_\epsilon) .
\end{align*}
The conclusion follows from equation \eqref{eq:descp-Kepsilon}.
\end{proof}

 Let $\phi : \mathbb{R} \to \mathbb{R}$ be any real-valued Lipschitz function. Define the smoothed version of $K^{\vartheta}_{n,\epsilon}$ by
\begin{equation}\label{eq:def-Kphi}
K^{\phi}_{n,\epsilon}(x\vert_0^{n-1}) \coloneqq \frac{1}{n^2} \sum_{i \neq j} \phi\left(1 - \frac{ d (x_i, x_j)}{\epsilon} \right).
\end{equation}
Observe that $K^{\phi}_{n,\epsilon}\in\Lip_d(M^n,\gamma\vert_{0}^{n-1})$ with \[
\gamma_k=\frac{2\operatorname{L}(\phi)}{n\epsilon},\quad k=0,\ldots,n-1.
\]
A standard Lipschitz approximation to the Heaviside function is $\phi_0 : \mathbb{R} \to [0,1]$ given by
\[
\phi_0(y) \coloneqq 
\begin{cases}
0 & \text{for } y < -\frac{1}{2}, \\
\frac{1}{2} + y & \text{for } -\frac{1}{2} \leq y \leq \frac{1}{2}, \\
1 & \text{for } y > \frac{1}{2}.
\end{cases}
\]
For every $y \in \mathbb{R}$, we easily check that
\begin{equation}\label{eq:ineq-phi0}
\vartheta(1 - 2y) \leq \phi_0(1 - y) \leq \vartheta(1 - y/2).
\end{equation}
This implies that, for all $\epsilon > 0$ and $n \geq 1$,
\begin{equation}\label{eq:comparison-kernels}
K^{\vartheta}_{n,\epsilon/2}\leq K^{\phi_0}_{n,\epsilon} \leq K^{\vartheta}_{n,2\epsilon}.
\end{equation}
It follows that, whenever $\dim_{\mathrm{c}}(\eta) > 0$, the asymptotic behavior
\[
K^{\vartheta}_{n,\epsilon}(X^x\vert_0^{n-1}) \approx \epsilon^{\dim_{\mathrm{c}}(\eta)} \quad \text{as } \epsilon \to 0,
\]
is equivalent to
\[
K^{\phi_0}_{n,\epsilon}(X^x\vert_0^{n-1}) \approx \epsilon^{\dim_{\mathrm{c}}(\eta)} \quad \text{as } \epsilon \to 0.
\]

Let us now show a concentration inequality for $K^{\phi_0}_{n,\epsilon}$:
\begin{theorem}
Let $\nu \in \prob(C(M))$, and suppose that the RDS induced by $\nu$ is weakly contracting on average. Let $\eta \in \prob(M)$ denote the $\nu$-stationary measure, and let $\lambda_\nu$ and $|\cF|_\infty$ be defined as in \eqref{eq:def-lambdanu} and \eqref{eq:def de Dnu}, respectively. 
Let $(X_n)_{n\geq 0}$ denote the associated fiber Markov chain defined in \eqref{eq:def-processXn}. 

Then, for any Lipschitz function $\phi : \mathbb{R} \to \mathbb{R}$, any $\epsilon > 0$, $n \in \mathbb{N}$, and all
\[
t > \frac{8\,\operatorname{L}(\phi)\,\lambda_\nu}{\epsilon n} + \frac{1}{n} \|\phi\|_\infty
\]
and every $x \in M$, the following concentration inequality holds:
\[
\bP\left( \left| K^{\phi}_{n,\epsilon}(\left.X^x\right|_0^n) - \int_{M^2} \phi_\epsilon\, \mathrm{d}(\eta \otimes \eta) \right| > t \right)
\le 2 \exp\left( -c\, n t^2 \epsilon^2 \right),
\]
where 
\[
c = \left( 192\, [\operatorname{L}(\phi)]^2 (|\cF|_\infty + \lambda_\nu)^2 \right)^{-1}
\quad \text{and} \quad
\phi_\epsilon(x, y) = \phi\left(1 - \frac{d(y, x)}{\epsilon}\right).
\]
\end{theorem}
\begin{proof}
Consider a Lipschitz function $\phi : \mathbb{R} \to \mathbb{R}$.
By Theorem \ref{thm:MAIN}, for every $x\in M$ and any $t>0$ we have, for all $n\in\bN$ and any $\epsilon>0$ 
\begin{align}\label{eq:dimcor-1}
&\bP\left( \left\vert K^{\phi}_{n,\epsilon}(\left.X^x\right\vert_0^n)-\bE \left[ K^{\phi}_{n,\epsilon}(\left.X^x\right\vert_0^n)\right] \right\vert> t\right) \nonumber\\
&\le \exp\left(-\frac{nt^2\epsilon^2}{48 [\operatorname{L}(\phi)]^2(|\cF|_\infty +\lambda_\nu)^2}\right).
\end{align}
Note that
\begin{align}\label{eq:dimcor-2}
\left\vert \bE \left[ K^{\phi}_{n,\epsilon}(\left.X^x\right\vert_0^n)\right]-\int_M \bE \left[ K^{\phi}_{n,\epsilon}(\left.X^y\right\vert_0^n)\right]\dd\eta(y)\right\vert\leq
\frac{2\operatorname{L}(\phi)}{n\epsilon}\lambda_\nu.
\end{align}
On the other hand, the $\nu$-stationarity of $\eta$ implies that
\begin{align*}
&\int_M \bE \left[ K^{\phi}_{n,\epsilon}(\left.X^y\right\vert_0^n)\right]\dd\eta(y)\\
&\quad=
\sum_{j=1}^{n-1}\frac{2(n-j)}{n^2}\int_M \bE \left[\phi\left(1 - \frac{ d (y, X_j^y)}{\epsilon} \right)\right]\dd\eta(y).
\end{align*}
Now, the $\nu$-stationarity of $\eta$ implies that 
\begin{align*}
&\int_M\int_M \phi\left(1 - \frac{ d (y, x)}{\epsilon} \right)\dd\eta(y)\dd\eta(x)\\
&=\frac{2}{n^2}\sum_{j=1}^{n-1}(n-j)\int_M\int_M \bE\left[\phi\left(1 - \frac{ d (y, X_j^x)}{\epsilon} \right)\right]\dd\eta(y)\dd\eta(x)\\
&\quad\quad\quad\quad +\frac{1}{n}\int_M\int_M \phi\left(1 - \frac{ d (y, x)}{\epsilon} \right)\dd\eta(y)\dd\eta(x).
\end{align*}
Thus,
\begin{align}\label{eq:dimcor-3}
    &\left\vert\int_M \bE \left[ K^{\phi}_{n,\epsilon}(\left.X^y\right\vert_0^n)\right]\dd\eta(y)-\int_M\int_M \phi\left(1 - \frac{ d (y, x)}{\epsilon} \right)\dd\eta(y)\dd\eta(x)\right\vert\nonumber\\
    &\leq \frac{2\operatorname{L}(\phi)}{\epsilon n^2}\sum_{j=1}^{n-1}(n-j)\int_M\int_M\bE\left[d(X_j^x,X_j^y)\right]\dd\eta(y)\dd\eta(x)+\frac1n\Vert\phi\Vert_\infty\nonumber\\
    &\leq \frac{2\operatorname{L}(\phi)\lambda_\nu}{\epsilon n}+\frac1n\Vert\phi\Vert_\infty
    \end{align}
Therefore, using triangle inequality and \eqref{eq:dimcor-1},\eqref{eq:dimcor-2},\eqref{eq:dimcor-3}, we get that for $t>\frac{4\operatorname{L}(\phi)\lambda_\nu}{\epsilon n}+\frac1n\Vert\phi\Vert_\infty$
\begin{align*}
&\bP\left( \left\vert K^{\phi}_{n,\epsilon}(\left.X^x\right\vert_0^n)-\int_M\int_M \phi\left(1 - \frac{ d (y, x)}{\epsilon} \right)\dd\eta(y)\dd\eta(x) \right\vert> 2t\right) \\
&\le 2\exp\left(-\frac{nt^2\epsilon^2}{48 [\operatorname{L}(\phi)]^2(|\cF|_\infty +\lambda_\nu)^2}\right).
\end{align*}
Rescaling $t$, we conclude the proof.
\end{proof}

\section{Examples}\label{sec:examples}
In this section, we analyze two concrete classes of random dynamical systems that satisfy the hypotheses of our main results. 

\subsection{RDSs on the circle}\label{sec:RDScircle}
Throughout this section, we assume that $\nu$ is a probability measure on the space of homeomorphisms on the circle $\bS^1$. We equip $\bS^1$ with the usual metric $d(x,y)=\min\{\lvert x-y\rvert,1-\lvert x-y\rvert\}$. We begin by introducing some concepts that will be used repeatedly in this section.

We say that the topological support $\cF$ of $\nu$ \emph{has a finite orbit} if there exists a finite set $\{x_1,\ldots,x_m\}$ of points such that for all $f\in\cF$, $f(\{x_1,\ldots,x_m\})\subset\{x_1,\ldots,x_m\}$.

The action of a semigroup $\cT$ of continuous functions $f:\bS^1\to\bS^1$ is called \emph{proximal} on a set $B\subset\bS^1$, if, for all $x,y\in B$, there exists a sequence $(f_n)_{n\in\bN}$ such that
    \[
    \lim_{n\to\infty}d(f_n(x),f_n(y))=0.
    \]
We say that $\nu$ is \emph{proximal} if the semigroup $\cT_\nu$ generated by its support is proximal on $\bS^1.$

We say that $\nu$ has the \emph{local contraction property} if there exists $q \in (0,1)$ such that for every $x \in \bS^1$, $\bP$-almost surely there exists an open neighborhood $B$ of $x$ for which
\begin{equation}\label{eq:defLC}
|G_n(B)|_d \le q^n,
\end{equation}
where $G_n$ denotes the random walk associated with the RDS induced by $\nu$, as defined in \eqref{eq:defGn}. Systems satisfying this property have attracted significant interest recently; see, for example, \cite{GelSal:23} and \cite{BarMal:24}, where the convergence in \eqref{eq:defLC} need not be exponential.

\subsubsection{Locally contracting on average}
Our first result on RDSs over the circle provides the main motivation for formulating the main results (see Section \ref{subsub:mainresults}) in terms of a family of sets  $\cI_1,\ldots,\cI_\ell$, which, in this section, will be assumed to be disjoint, closed, and connected.
\begin{proposition}\label{prop:CAlocal}
    Let $\nu$ be a finitely supported probability measure on the space of $C^2$-diffeomorphisms on $\bS^1$,  with support $\cF$. Suppose that $\cF$ has no finite orbit.  Let $(X_n)_{n\geq 0}$ denote the associated fiber Markov chain defined in \eqref{eq:def-processXn}. Assume there exist $\ell\in\bN$ and $\ell$ disjoint closed connected sets $I_1,\ldots,I_\ell\subset \bS^1$ such that the semigroup $\cT_\nu$ generated by $\cF$ acts proximally on each $I_i$, and for every $i\in\{1,\ldots,\ell\}$ and $f\in\cF$, there exist $j,k\in\{1,\ldots,\ell\}$ with $f(I_k)\subset I_i$ and $f(I_i)\subset I_j$. Then, there exist $\alpha,r\in(0,1)$ and $c>0$ such that for all $n\in\bN$
    \[
    \sup_{i\in\{1,\ldots,\ell\}}\sup_{x,y\in I_i\colon x\neq y}\frac{\bE \left[d^\alpha(X_n^x,X_n^y)\right]}{d^\alpha(x,y)}\leq c\, r^n.
    \]
    In particular, we have
    \[
    \sup_{i\in\{1,\ldots,\ell\}}\sup_{x,y\in I_i\colon x\neq y}\sum_{n=0}^{\infty}\bE \left[d(X_n^x,X_n^y)\right]\leq \frac{c}{1-r}.
    \]
\end{proposition}

Before proving Proposition \ref{prop:CAlocal}, let us make two important observations.

\begin{remark}
    The assumption of the absence of finite orbits in Proposition \ref{prop:CAlocal} ensures that the sets $I_i$ are non-degenerate. One might consider removing this assumption; however, in that case, we allow the possibility that $I_j=\{z_j\}$ for some $z_j\in\bS^1$,  which, due to the properties of the sets $I_i$, would imply that  $I_i=\{z_i\}$ for some $z_i\in \bS^1$ for each $i\in\{1,\ldots,\ell\}$. In such a case, the conclusion of Proposition \ref{prop:CAlocal} becomes trivial.
\end{remark}

\begin{remark}\label{rem:justificaionMR}
    A large class of examples of RDSs of circle diffeomorphisms satisfy the hypothesis in Proposition \ref{prop:CAlocal}. By \cite{Mal:17}, an RDS without finite orbits has only two options: either $\cT_\nu$ is semiconjugated to the symmetry group over $\bS^1$ or there is no probability measure on $\bS^1$ that is invariant for each map in $\cT_\nu$ (so, $\nu$ has the local contraction property). By assuming proximality over a neighborhood, we must be in the case of local contraction. Therefore, we can be in some of the following three cases (which cover a large class of examples):
\begin{itemize}
    \item In the case of a unique minimal set (and so, the uniqueness stationary measure):
    \begin{itemize}
        \item We can have global proximality, that is, $I_1=\bS^1$ (and so, $\ell=1$) , which was studied in \cite{GelSal:23}.
        \item There are examples where such a finite family of closed connected sets exists. For instance, \cite[Example 1]{MalSal:25}  presents an RDS with a unique minimal set, yet there exist two disjoint closed connected sets $I_1, I_2$ that satisfy the properties stated in Proposition \ref{prop:CAlocal} for $\ell=2$.
    \end{itemize}
    \item In the case of an RDS with at least two minimal sets, \cite[Theorem 1]{MalSal:25} guarantees the existence of a family of closed connected sets as described in Proposition \ref{prop:CAlocal}. Moreover, the cardinality of such families coincides with the (finite) number of minimal sets.
\end{itemize}
\end{remark}

Consider the collection $\{I_1,\ldots,I_\ell\}$ as in Proposition \ref{prop:CAlocal}.
Let $\eta$ be the $\nu$-stationary measure supported on $\cup_{i=1}^\ell I_i$. The existence and uniqueness of such a measure are discussed in \cite{MalSal:25}.
The \emph{Lyapunov exponent} associated to $\eta$ is defined as
\begin{equation}\label{eq:def-LyapEXP}
L(\eta)\coloneqq\int_{\bS^1}\int_\cF \log\vert f'(x)\vert\,\dd\nu(f)\,\dd\eta(x).
\end{equation}
Under the assumptions of Proposition \ref{prop:CAlocal}, the results in \cite{Mal:17} apply, and so there is a unique stationary measure $\eta$ supported on $\cup_{i=1}^{\ell}I_i$. Moreover, we have $L(\eta)<0$. Further, for all $i\in\{1,\ldots,\ell\}$, $x,y\in I_i$, we have $\bP$-almost surely
    \begin{equation}\label{eq:prop:CAlocal-01}
        \limsup_{n\to\infty}\frac1n\log d (X_n^x,X_n^y)\leq L(\eta).
        \end{equation}        
        and there exists an open neighborhood $B$ of $x$ such that
        \begin{equation}\label{eq:prop:CAlocal-02}
        \limsup_{n\to\infty}\frac1n\log \vert G_n(B)\vert_d\leq L(\eta),
        \end{equation} 

\begin{proof}[Proof of Proposition \ref{prop:CAlocal}]
Let us write $L$ instead of $L(\eta)$ in this proof.
    We follow the proof of \cite[Theorem 1.3]{GelSal:23}, considering the subadditive process 
    \begin{equation}\label{eq:prop:CAlocal-03}
    b_n=\max_{i\in\{1,\ldots,\ell\}}\,\sup_{x,y\in I_i\colon x\neq y}\,\bE\left[\log\frac{d(X_n^x,X_n^y)}{d(x,y)}\right].
    \end{equation} 
    Let us sketch this proof. Since the invariance of the sets $I_i$, we can conclude $(b_n)_{n\in\bN}$ is a subadditive sequence. Hence, by Fekete’s Lemma, the limit $b=
\lim_{n\to\infty}b_n/n = \inf_{n\in\bN} b_n/n \in [-\infty,\infty)$ exists. To show that $b\leq L$, assume by contraction that there exist two sequences $(x_n)_{n\in\bN}$ and $(y_n)_{n\in\bN}$ such that for all $n\in\bN$ there is some $i_n\in\{1,\ldots,\ell\}$ such that $x_n,y_n\in I_{i_n}$ and
\[
L<L'\leq\frac1n\bE\left[\log\frac{d(X_n^{x_n},X_n^{y_n})}{d(x_n,y_n)}\right],
\]
for some $L'>L$. By compactness, there exist a subsequence $(n_k)_{k\in\bN}$, an index $i\in\{1,\ldots,\ell\}$ and points $x, y \in I_i$ such that $i_{n_k}=i$, $\lim_k x_{n_k}=x$ and $\lim_k y_{n_k}=y$. Let us write $x_k=x_{n_k}$ and $y_k=y_{n_k}$. Using the differentiability of maps in $\cF$, the fact that $\bP$-almost surely \eqref{eq:prop:CAlocal-01} and \eqref{eq:prop:CAlocal-02} hold, and dominated convergence theorem, we can conclude that
\[
\lim_{k\to\infty}\frac{1}{n_k}\bE\left[\log\frac{d(X_{n_k}^{x_k},X_{n_k}^{y_k})}{d(x_k,y_k)}\right]\leq L
\]
which contradicts \eqref{eq:prop:CAlocal-03}. Therefore, $b\leq L.$ 

Now, fix $n\in\bN$ sufficiently large, $\frac1n b_n\leq \frac12 L<0$. Since $\nu$ is finitely supported, there exists $c>0$ such that for all $x,y\in \bS^1$ and $n\in\bN$
\[
-c\leq \frac1n\log\frac{d(X_n^{x_n},X_n^{y_n})}{d(x_n,y_n)}\leq c
\]
Using the above and that $e^x\leq 1+ x +x^2\e^{|x|}$, for every $\alpha\in(0,1)$ it follows that for all $i\in\{1,\ldots,\ell\}$ and $x,y\in I_i$, $x\neq y$, we have
\begin{align*}
    \bE\left[\frac{d^\alpha(X_n^{x},X_n^{y})}{d^\alpha(x,y)}\right]&=\bE\left[\exp\left(\alpha\log\frac{d(X_n^{x},X_n^{y})}{d(x,y)}\right)\right]\\
    &\leq  1+\alpha\bE\log\frac{d(X_n^{x},X_n^{y})}{d(x,y)}+\alpha^2(nc)^2\e^{nc}\\
    &\leq 1+\alpha b_n+\alpha^2(nc)^2\e^{nc}\\
    &\leq 1+\alpha \frac{n}{2} L+\alpha^2(nc)^2\e^{nc}.
\end{align*}
Hence,
\[
\max_{i\in\{1,\ldots,\ell\}}\,\sup_{x,y\in I_i\colon x\neq y}\bE\left[\frac{d^\alpha(X_n^{x},X_n^{y})}{d^\alpha(x,y)}\right]\leq 1+\alpha \frac{n}{2} L+\alpha^2(nc)^2\e^{nc}.
\]
Now, taking $\alpha\in(0,1)$ sufficiently small, the right-hand side provides a contraction rate in
$(0, 1)$. Use the fact that the process on the right in the inequality above is submultiplicative in $n$ to complete the first part. For the second part, use that $d\leq d^\alpha\leq 1$.
\end{proof}

\subsubsection{Concentration inequalities on the circle}

Before stating the next result, we note that its proof is omitted, as it follows directly from the arguments developed throughout the paper. Specifically, it is obtained by combining the local contraction estimates on the circle with the general concentration inequalities established earlier.
In fact, the theorem below can be viewed as a version of Theorem~\ref{Thm:LG}, where a law of large numbers for Birkhoff sums emerges as a direct consequence of quantitative concentration estimates.

Recall the definition of $S_n^x(h)$ given in \eqref{eq-def:Snx}. 

\begin{theorem}\label{thm:circleLLN}
   Assume the hypotheses of Proposition \ref{prop:CAlocal}. Let $\eta$ be the $\nu$-stationary measure supported on $\cup_{i=1}^\ell I_i$. For every Lipschitz function $h\colon\bS^1\to\bR$, the sequence $(t_n)_{n\in\bN}$ given by
   \[
   t_n=\sup_{x\in I_1\cup\cdots\cup I_\ell}\left|\frac1n \bE [S_n^x(h)]-\eta(h)\right|
   \]
   converges to 0. And, for all $n\in\bN$, $t>2 t_n|$ and $x\in \cup_{i=1}^\ell I_i$, the following inequality holds
    \[
\bP\left( \left|\frac1n S_n^x(h)-\eta(h)\right|> t\right) \le 2\exp\left(\frac{-\, n\,t^2}{48\operatorname{L}(h)^2(\vert\cF\vert_\infty+\lambda)}\right),
\]
where
\[
\lambda=\sup_{i\in\{1,\ldots,\ell\}}\sup_{x,y\in I_i\colon x\neq y}\sum_{n=0}^{\infty}\bE \left[d(X_n^x,X_n^y)\right].
\]
\end{theorem}

For the following result, consider the metric $\varrho$ on the space of $C^2$-diffeomorphisms on $\bS^1$ given by
\[
\varrho(f,g)=\sup_{x\in\bS^1}\left(d(f(x),g(x))+\vert f'(x)-g'(x)|\right).
\]
\begin{theorem}\label{thm:circle-LE}
    Assume the hypotheses of Proposition \ref{prop:CAlocal}. Let $\eta$ be the $\nu$-stationary measure supported on $\cup_{i=1}^\ell I_i$. Set
\[
m_\nu=\min_{f\in\cF}\Vert f'\Vert_\infty\quad\text{and}\quad M_\nu=\max_{f\in\cF}\Vert f'\Vert_\infty.
\]
Then, there exists a sequence $(t_n)_{n\in\bN}$ of positive numbers converging to 0, such that for all $n\in\bN$, $t>t_n$ and $x\in \cup_{i=1}^\ell I_i$, the following inequality holds:
     \[
\bP\left( \left|\frac1n\log |G_n'(x)|-L(\eta)\right|> t\right) \le 2\exp\left(-\frac{nt^2m_\nu^2}{48M_\nu^2 (\vert\cF\vert_\varrho+\lambda)^2}\right),
\]
where $L(\eta)$ is as in \eqref{eq:def-LyapEXP}, and $G_n$ is the random process defined in \eqref{eq:defGn}.
\end{theorem}
\begin{proof}
    Consider $\varphi\colon(\bS^1)^n\times \cF^n\to\bR$ by
    \[
    \varphi(x\vert_{0}^{n-1},f\vert_1^n)=\frac{1}{n}\log\left\vert \prod_{k=1}^{n} f_k'(x_{k-1})\right\vert.
    \]
    Then, using that 
    \[
|\log |u| - \log|v| | = \left| \log \left( \frac{|u|}{|v|} \right) \right| \leq \frac{|u - y|}{\min(|u|, |v|)}.
\]
we get,
    \begin{align*}
        &\vert\varphi(x\vert_{0}^{n-1},f\vert_1^n)-\varphi(y\vert_{0}^{n-1},g\vert_1^n)\vert\\
        &\leq\frac{1}{n\,m_\nu}\sum_{k=1}^{n}\vert f_k'(x_{k-1})-g_k'(y_{k-1})\vert\\
        &\leq \frac{1}{n\,m_\nu}\sum_{k=1}^{n}\left(\vert f_k'(x_{k-1})-f_k'(y_{k-1})\vert+\vert f_k'(y_{k-1})-g_k'(y_{k-1})\vert\right),
    \end{align*}
    and, since $m_\nu\leq 1\leq M_\nu$, we get
    \begin{align*}
        \vert\varphi(x\vert_{0}^{n-1},f\vert_1^n)-\varphi(y\vert_{0}^{n-1},g\vert_1^n)\vert\leq \frac{M_\nu}{n\,m_\nu}\sum_{k=1}^{n}\left(d(x_{k-1},y_{k-1})+\varrho(f_k,g_k)\right).
    \end{align*}
    Therefore, $\varphi\in \Lip_{d+\varrho}((\bS^1)^n\times\cF^n, \gamma\vert_0^{n-1})$ with $\gamma_k=\frac{M_\nu}{n\,m_\nu}$ for each $k\in\{0,\ldots,n-1\}.$
    Note that for all $x\in\bS^1$ we have
    \[
    \varphi(X^x\vert_0^{n-1},F\vert_1^n)=\frac{1}{n}\log \vert G_n'(x)\vert.
    \]
    By Proposition \ref{prop:CAlocal}, 
    \[
    \lambda=\sup_{i\in\{1,\ldots,\ell\}}\sup_{x,y\in I_i\colon x\neq y}\sum_{n=0}^{\infty}\bE [d(X_n^x,X_n^y)]\leq \frac{c}{1-r}<\infty.
    \]
    By Theorem \ref{thm:MAIN}, for all $t>0$ and $x\in\bS^1$ we get
    \[
    \bP\left(\frac{1}{n}\log \vert G_n'(x)\vert-\frac{1}{n}\bE\left[\log \vert G_n'(x)\vert\right]>t\right)\leq 2\exp\left(-\frac{nt^2m_\nu^2}{12M_\nu^2 (\vert\cF\vert_\varrho+\lambda)^2}\right)
    \]
Set
\[
\hat t_n=\sup_{x\in I_1\cup\cdots\cup I_\ell}\left|\frac1n\bE\left[\log \vert G_n'(x)\vert\right]-L(\eta)\right|.
\]
By \cite[Proposition 4.11]{Mal:17}, $\hat t_n\to 0$ as $n\to\infty$. Set $t_n=2\hat t_n$ and $t=2s$ to conclude this proof.
\end{proof}

\subsubsection{Synchronization with non-expansive points}
From a topological perspective, the dynamics of an RDS on the circle is closely related to the action of the semigroup (or group) generated by its constituent homeomorphisms; see, e.g., \cite{Navas:2011,MalSal:25}. In the context of group actions, non-expansive points under such group actions have been shown to play a key role in understanding ergodic properties, as illustrated in \cite{DerKleNav:2009}.

   For each $n \in \mathbb{N}$, we define $\NF(n)$ as the (random) set of fixed, non-expansive points of the map $G_n$.
More precisely, for each realization $\omega \in \Omega$, we set
\[
\NF(n)[\omega]
\coloneqq
\left\{
x \in M \;:\;
G_n[\omega](x) = x
\quad \text{and} \quad
\lvert (G_n[\omega])'(x) \rvert \leq 1
\right\},
\]
where 
\[
(G_n[\omega])'(x)=(f_1\circ\cdots \circ f_n)'(x),\quad\omega=(f_1,f_2,\ldots).
\]
As customary, we will often write $\NF(n)$ and $G_n$, leaving the dependence on $\omega$ implicit.

   \begin{theorem}
Let $\nu$ be a finitely supported probability measure on the space of $C^2$ -diffeomorphisms on $\bS^1$. Assume the support $\cF$ of $\nu$ has no finite orbits and that $\nu$ is proximal. Then there exist $c>0$ and a sequence $(t_n)_{n\in\bN}$ of positive numbers converging to 0, such that for any $n\in\bN$, $x\in\bS^1$ and $t>t_n$
\[
\bP\left( \frac1n \inf_{y\in \NF(n)}\sum_{k=0}^{n-1}d( X_k^x, X_k^y)>t\right)\leq \e^{-c\,n\,t^2}.
\]
\end{theorem}
\begin{proof}
Let us introduce the maps $(\widehat G_n)_{n\in\mathbb{N}}$, defined by
\begin{equation}\label{eq:defGhat}
   \widehat G_n[\omega] \coloneqq f_1 \circ \cdots \circ f_n,
   \quad n \in \mathbb{N}, \quad \omega = (f_1,f_2,\dots) \in \Omega,
\end{equation}
with the convention $\widehat G_0 = \operatorname{id}_M$. As for $G_n$, we will often write $\widehat G_n(x)$, leaving the dependence on $\omega$ implicit.
Whenever necessary, it can be made explicit by writing $\widehat G_n[\omega](x)$.

   Recall the definition of the fiber Markov chain in \eqref{eq:def-processXn} associated to $\nu$, to see that $G_n(x)=X_n^x$.
Since $G_n$ and $\widehat G_n$ are equally distributed for each $n\in\bN$ and $x\in\bS^1$, we have
\begin{align*}
&\bP\left(\frac1n \inf_{y\in \NF(n)}\sum_{k=0}^{n-1}d(G_k(x),G_k(y)) \right)\\&=
\bP\left( \frac1n \inf_{y\in \hat\NF(n)}\sum_{k=0}^{n-1}d(\widehat G_k(x),\widehat G_k(y)) \right).
\end{align*}
Hence, let us prove the upper bound for the value on the right side in the equality above.
   
From Proposition \ref{prop:CAlocal} (or, \cite[Theorem 1.3]{GelSal:23}) and \cite[Theorem 1]{Sten:2012}, there exist a measurable map $Z:\Omega\to\bS^1$ and constants $c>0$ and $r\in(0,1)$ such that for all $x,y\in\bS^1$ and $n\in\bN$
\[
\bE[d(\widehat G_n(x),Z)]\leq c r^n\quad\text{and}\quad\bE[d(\widehat G_n(x),\widehat G_n(y))]\leq c r^n.
\]
Moreover, as discussed in \cite{SalBez:25}, $\bP$-almost surely there exists a non-degenerate connected neighborhood $B$ of $Z$ such that for all $x\in B$ the sequence $(\widehat G_n(x))_{n\in\bN}$ converges to $Z$ and, furthermore,
\[
\lim_{n\to\infty}|\widehat G_n (B)|_d=0.
\]
The last convergence occurs at an exponential rate.
Therefore, using the topological structure of $\bS^1$ there exists $N\in\bN$ such that for all $n\geq N$ we have $\widehat G_n (B)\subsetneq B$ and $\widehat G_n (B)\cap \widehat \NF(n)\neq\emptyset$. So, for $n\geq N$,
\[
\frac1n \inf_{y\in \hat\NF(n)}\sum_{k=0}^{n-1}d(Z,\widehat G_k(y))\leq \frac1n\sum_{k=0}^{n-1}|\widehat G_k (B)|_d,
\]
Hence, $\bP$-almost surely
\[
\lim_{n\to\infty}\frac1n \inf_{y\in \hat\NF(n)}\sum_{k=0}^{n-1}d(Z,\widehat G_k(y))=0.
\]
By the dominated convergence theorem,
\[
\lim_{n\to\infty}\frac1n \bE\left[\inf_{y\in \hat\NF(n)}\sum_{k=0}^{n-1}d(Z,\widehat G_k(y))\right]=0.
\]
Set
\[
t_n=\frac{c}{n(1-r)}+\frac1n \bE\left[\inf_{y\in \hat\NF(n)}\sum_{k=0}^{n-1}d(Z,\widehat G_k(y))\right].
\]
Note that for all $x\in \bS^1$,
\[
\frac1n \bE\left[\inf_{y\in \hat\NF(n)}\sum_{k=0}^{n-1}d(\widehat G_k(x),\widehat G_k(y))\right]\leq t_n.
\]
Proceeding as in the proof of Theorem \ref{thm:shadowing}, we can conclude that there exists $c>0$ such that for any $n\in\bN$, $x\in\bS^1$ and $t>t_n$
\[
\bP\left( \frac1n \inf_{y\in \hat\NF(n)}\sum_{k=0}^{n-1}d(\widehat G_k(x),\widehat G_k(y))>2t\right)\leq \e^{-c\,n\,t^2}.
\]
By rescaling $t$ we can conclude this proof.
\end{proof}

\subsection{RDSs on the projective space}\label{sec:RDSproj}
Let $P^{m-1}$ be the projective space of $\bR^m$ and consider the following projective distance 
\begin{equation*}
d(x,y)\coloneqq \Vert x \wedge  y\Vert=|\sin\measuredangle(x,y)|, \quad x,y\in P^{m-1}.
\end{equation*}
Each point in $x\in P^{m-1}$ represents a direction (i.e., a line through the origin) in $\bR^m$. Without loss of generality, we represent each point $x \in \mathbb{P}^{m-1}$ by a unit vector in $\mathbb{R}^m$.
 For a matrix $A\in \SL(m,\bR)$, define the \emph{projective map}
\begin{equation}\label{def:fA-proj-map}
f_{A}:P^{m-1}\to P^{m-1},\quad f_{A} (x)\coloneqq \frac{A x}{\Vert A x \Vert}.
\end{equation}
\begin{remark}
    Since we are assuming that $A\in \SL(m,\bR)$, a unique matrix determines each projective map, that is, if $f_A=f_B$ for some $A,B\in \SL(m,\bR)$, then $A=B$.
\end{remark}

Given a measure $\nu$ supported in the space of the projective maps. Consider the sequence $(\bA_n)_{n\geq 0}$ of \emph{random matrices} in $\SL(m,\bR)$ such that the random projective map $G_n$ (associated to $\nu$, see \eqref{eq:defGn}) is determined by $\bA_n$, that is, $G_n=f_{\bA_n}$. Note that $\bA_0$ is the identity matrix. We can define each $\bA_n$ as the random product of $n$ independent matrices (with the same distribution as that induced by $\nu$), but this is not our interest here. 

If the RDS induced by $\nu$ is weakly contracting on average and $\eta$ is the $\nu$-stationary measure, then by Kingman ergodic theorem, we have for $\eta$-almost every $y\in P^{m-1}$
\begin{equation*}
    \Lambda_\nu\coloneqq\int_{P^{m-1}}\bE\left[\log \Vert \bA_1 x\Vert\right]\,\dd\eta(x)=\lim_{n\to\infty}\frac{1}{n}\bE\left[\log \Vert \bA_n y\Vert\right].
\end{equation*}

The following result provides an alternative proof of the large deviation bounds for Lyapunov exponents obtained in \cite[Section 5]{DuaKle:2016}. In contrast with their approach, which relies on a spectral analysis of the Markov operator associated with the RDS, our method is based on a concentration inequality for separately Lipschitz observables.
\begin{theorem}\label{thm:ConcLinLyap}
Let $m\geq 2$. Consider the projective space $P^{m-1}$ of $\bR^m$.
    Let $\nu$ be a probability measure on the space of all projective maps $f_A\colon P^{m-1}\to P^{m-1}$ as in \eqref{def:fA-proj-map} with $A\in \SL(m,\bR)$. Let $\cF$ be the topological support of $\nu$. Assume there exists $C>0$ such that if $f_A\in\cF$ for some $A\in\SL(m,\bR)$ then $\max\{\Vert A\Vert,\Vert A^{-1}\Vert \}\leq C$. Let $(\bA_n)_{n\geq 0}$ be the sequence of random matrices in $\SL(m,\bR)$ associated to $\nu$. For all $x\in P^{m-1}$, consider the fiber Markov chain $(X_n^x)_{n\geq 0}$ associated to $\nu$. 
    Assume that the RDS induced by $\nu$ is weakly contracting on average and let $\lambda_\nu$ be defined as in \eqref{eq:def-lambdanu}.
    Then the sequence $(t_n)_{n\in\bN}$ given by
    \[
    t_n=\sup_{x\in P^{m-1}}\left\vert\frac{1}{n}\bE\left[\log \Vert \bA_n x\Vert\right]-\Lambda_\nu\right\vert,
    \]
    converges to 0 as $n\to\infty.$ Moreover, for all $n\in \bN$ and $t>2t_n$
    \[
    \bP\left( \left\vert\frac1n \log \Vert \bA_n x\Vert-\Lambda_\nu\right\vert > t\right) \le \exp\left(-\frac{t^2}{192\, C^4 (\lambda_\nu+C)^2}\right).
    \]
\end{theorem}

The proof is omitted, since it follows closely the ideas of the proof of Theorem~\ref{thm:circle-LE}.

\begin{corollary}\label{cor:concentEXPLINEAR}
    Under the conditions of Theorem \ref{thm:ConcLinLyap}, for all $n\in\bN$  we have
    for $t>2t_n+\frac2n\log m$,
\[
\bP\left(\left\vert\frac1n \log \Vert \bA_n \Vert-\Lambda_\nu\right\vert>t\right)\leq 2m \exp\left(-\frac{t^2}{768\, C^4 (\lambda_\nu+C)^2}\right).
\]
\end{corollary}

The above corollary and Borel--Cantelli Lemma guarantee that $\Lambda_\nu$ coincides with the Lyapunov exponent of the linear cocycle associated with $\nu$, so that almost surely we have
\begin{equation*}
    \Lambda_\nu=\lim_{n\to\infty}\frac{1}{n}\log \Vert \bA_n \Vert.
\end{equation*}
This equality is well established by assuming some irreducibility condition on the associated cocycle. For example, in \cite[Corollary 1.3, Sec. III]{84:Ledrappier} it was shown for strongly irreducible cocycles and in \cite[Lemma 4.3]{DuaKle:2016} for quasi-irreducible cocycles.

\begin{proof}[Proof of Corollary \ref{cor:concentEXPLINEAR}]
Consider the canonical basis $\{e_1,\ldots,e_m\}$ of $\bR^N$. In our setting, $\{e_1,\ldots,e_m\}\subset P^{m-1}$. Using the euclidean structure of $\bR^m$, we get
\[
\max_{1\leq i \leq m} \Vert\bA_n e_i\Vert \leq \Vert \bA_n\Vert\leq m \max_{1\leq i \leq m}\Vert\bA_n e_i\Vert.
\]
Thus,
\[
\left\vert\frac1n \log \Vert \bA_n \Vert-\Lambda_\nu\right\vert\leq \max_{1\leq i \leq m} \left\vert\frac1n \log \Vert \bA_n e_i\Vert-\Lambda_\nu\right\vert+\frac1n\log m,
\]
and so, for $t>2t_n+\frac2n\log m$,
\[
\bP\left(\left\vert\frac1n \log \Vert \bA_n \Vert-\Lambda_\nu\right\vert>t\right)\leq \sum_{i=1}^{m} \bP\left(\left\vert\frac1n \log \Vert \bA_n \Vert-\Lambda_\nu\right\vert>\frac{t}{2}\right).
\]
The corollary follows from Theorem \ref{thm:ConcLinLyap}.
\end{proof}

\appendix
\section{}
We present here some technical results and auxiliary estimates used in the main text.
\begin{lema}\label{Appendix:pre-lemma01}
    For all $u\geq 0$, we have $p(u)=3u^2-\frac{3}{2}u+1\geq 0$.
\end{lema}
\begin{proof}
    Since $p'(u)=6u-\frac{3}{2}$, we have that $p$ is increasing on $[1/4,+\infty)$. Note that $p(1/4)=3/16-3/8+1=13/16>0$. Hence, $p(u)>0$ for all $u\geq 1/4$. For $u\in[0,1/4]$, we have
    \begin{align*}
        3u^2-\frac{3}{2}u+1\geq -\frac{3}{8}+1>0.
    \end{align*}
    The lemma is proved.
\end{proof}

\begin{lema}\label{Appendix:lemma01}
    For all $u\geq 0$, we have $1+(u^2\e^u)/2\leq \e^{3u^2}$.
\end{lema}
\begin{proof}
    For $u\geq 0$, set $f(u)=\e^{3u^2}-1-\frac{1}{2}u^2\e^u$. Then $f(0)=0$,
    and 
    \[
    f'(u)=6u\e^{3u^2}-u\e^u-\frac{1}{2}u^2\e^u=u\e^u\left(6\e^{3u^2-u}-1-\frac{u}{2}\right).
    \]
    To show the desired inequality, it is enough to prove that $f$ is increasing, so it suffices to show that for $u\geq 0$
    \begin{equation}\label{APP:lemma01-eq1}
    6\e^{3u^2-u}-1-\frac{u}{2}\geq 0.
    \end{equation}
Using the classical inequality $1+x\leq \e^x$ for $x\geq 0$, we get
\[
6\e^{3u^2-u}-1-\frac{u}{2}\geq 6\e^{3u^2-u}-\e^{\frac{u}{2}},
\]
so that \eqref{APP:lemma01-eq1} is consequence of 
\begin{equation}\label{APP:lemma01-eq2}
6\e^{3u^2-u}\geq \e^{\frac{u}{2}}.
\end{equation}
    Let us show \eqref{APP:lemma01-eq2}.
    In fact, since $p(u)=3u^2-\frac{3}{2}u+1\geq 0$ for all $u\geq 0$ (see Lemma \ref{Appendix:pre-lemma01}, we have
    \[
    3u^2-\frac{3}{2}u\geq -1\geq -\log 6,
    \]
    which implies \eqref{APP:lemma01-eq2}, and hence the desired result.
\end{proof}

\begin{lema}\label{lem:AP-varZ}
    For any $\bR$-valued random variable $Z$, we have 
    \[
    \bE \left[\mathbbm{1}_{(K, \infty)}(Z) Z\right]\leq \frac{\bE \left[Z^{2}\right]}{K},\quad\text{for } K>0.
    \]
\end{lema}
\begin{proof}
Applying Cauchy-Schwarz inequality and then Bienaymé-Chebychev inequality, we obtain
    \begin{align*}
\bE \left[\mathbbm{1}_{(K, \infty)}(Z) Z\right] & \leqslant \sqrt{\bE \left[\mathbbm{1}_{(K, \infty)}(Z)\right]} \sqrt{\bE \left[Z^{2}\right]}\\
&=\sqrt{\bP(Z>K)} \sqrt{\bE \left[Z^{2}\right]} \\
& \leqslant \frac{\sqrt{\bE \left[Z^{2}\right]}}{K} \sqrt{\bE \left[Z^{2}\right]}\\
&=\frac{\bE \left[Z^{2}\right]}{K} .
\end{align*}
This lemma is proved.
\end{proof}

\printbibliography

\end{document}